\documentclass[11pt]{article}
\usepackage{latexsym, amsfonts, amscd, amssymb, verbatim, amsmath,
enumerate}

\newtheorem{theorem}{Theorem}[section]

\newtheorem{lemma}{Lemma}[section]
\newtheorem{remark}{Remark}[section]

\newtheorem{example}{Example}[section]

\newcommand{\proofbox}{\hspace{\fill}{$\Box$}}
\newenvironment{proof}{\textbf{Proof}.}{\proofbox}

\def\bar{\overline}
\date{}

\numberwithin{equation}{section}

\begin{document}

\title{Regularity of multipliers and second-order optimality conditions for semilinear parabolic optimal control problems with mixed pointwise constraints} 
	
\author{H. Khanh\footnote{Department of Optimization and Control Theory,	Institute of Mathematics, Vietnam Academy of Science and Technology,	18 Hoang Quoc Viet road, Hanoi, Vietnam  and Department of Mathematics, FPT University, Hoa Lac Hi-Tech Park, Hanoi, Vietnam; email: khanhh@fe.edu.vn} \  and  B.T. Kien\footnote{Department of Optimization and Control Theory, Institute of Mathematics, Vietnam Academy of Science and Technology,18 Hoang Quoc Viet road, Hanoi, Vietnam; email: btkien@math.ac.vn}}
	
\maketitle
	
\medskip
	
\noindent {\bf Abstract.} {\small A class of  optimal control problems governed by semilinear parabolic equations with mixed pointwise constraints is considered.  We give some criteria under which the first and second-order  optimality conditions are of KKT-type. We then prove that the Lagrange multipliers belong to $L^p$-spaces. Moreover, we show that if the initial value is good enough and boundary $\partial\Omega$ has a property of positive geometric density, then multipliers and optimal solutions are H\"{o}lder continuous.}
	
\medskip
	
	
	
\vspace{0.4cm}

\section{Introduction}
	
Let $\Omega$ be a bounded domain in $\mathbb{R}^N$ with $N=2$ or $3$ and its  boundary $\Gamma=\partial \Omega$ is regular enough. Let $D=H^2(\Omega)\cap H_0^1(\Omega)$ and $H=L^2(\Omega)$. We consider  the problem of finding control function $u\in L^\infty(Q)$  and the corresponding state $y\in L^\infty(Q)\cap W^{1,2}(0, T; D, H)$ which solve
\begin{align}
    &J(y, u):=\int_Q L(x, t,  y(x, t), u(x, t)) dxdt\to\inf \label{P1}\\
    &{\rm s.t.}\notag\\
    &\frac{\partial y}{\partial t} +Ay + f(y) =u\quad \text{in}\quad Q= \Omega  \times (0, T)\label{P2} \\
    & y(x, t)=0 \quad \text{on}\quad \Sigma=\Gamma\times (0, T), \quad y(x, 0)=y_0(x) \quad \text{on}\quad \Omega \label{P3}\\
    & g(x, t, y(x, t),  u(x,t))\leq 0 \quad \text{a.a.}\quad (x, t)\in Q, \label{P4}
\end{align} where $y_0\in H^1_0(\Omega)\cap L^\infty(\Omega)$,  $A$ denotes a second-order elliptic operator of the form
$$
Ay =  - \sum_{i,j = 1}^N D_j\left(a_{ij}\left( x \right) D_i y \right),
$$  $L:Q \times {\mathbb R} \times {\mathbb R} \to {\mathbb R}$,  $f: {\mathbb R} \to {\mathbb R}$ and $g: \Omega\times[0, T]\times \mathbb{R}\times \mathbb{R}\to\mathbb{R}$ are of class $C^2$. 

The theory of optimal control governed by parabolic equations has studied by many mathematicians so far. The mainly research topics in this area consist of the existence of optimal solutions, optimality conditions and numerical methods for computing  optimal solutions. In order to compute approximate optimal solutions, we need to derive  optimality conditions.  Based on optimality conditions and regularity of multipliers and optimal solutions,  we can prove the convergence of approximate solutions and give error estimates. Therefore, the optimality conditions and regularity multipliers and optimal solutions play an important role in numerical methods such as the finite element method (FEM).  Recently, there have been many papers dealing with optimality conditions for optimal control problems governed by parabolic equations. For papers which have close connection to the present work, we refer the reader to \cite{Arada-2002-1}, \cite{Arada-2002-2},  \cite{Casas-2000}, \cite{Raymond} and \cite{Troltzsch}, and many references given therein. Among these papers, Arada and Raymond \cite{Arada-2002-2},  Raymond and Zidani \cite{Raymond} focused on first-order optimality conditions including Prontryagin's principle.  Casas and F. Tr\"{o}ltzsch \cite{Casas-2015} studied second-order optimality conditions for problems where the objective functional is of quadratic form which may not  include a Tikhonov regularization term. 

In this paper, we are going to establish first and second-order  optimality conditions and study regularity of multipliers for  optimal control problem \eqref{P1}-\eqref{P4} which is governed by semilinear parabolic equations with mixed pointwise constraints. We give some criteria under which the  optimality condition is of KKT-type. We then prove that the Lagrange multipliers belong to $L^p$-spaces. It is interesting to see that, due to the parabolic nature of the problem, we do not need the Slater condition for  the existence of multipliers. We shall show that if  the linearization of the constraints (state equation
together with the mixed constraint) is surjective, then there exist regular multipliers. This has done in \cite{Meyer} and \cite{Neitzel} for elliptic optimal control problems and linear parabolic optimal control problems.  However, in our problem, the situation is more complicated because both the state equation and the constraints are nonlinear, and the control variable belongs to $L^\infty$. Therefore, it is harder to establish optimality conditions and regularity of multipliers.  When the control variable $u$ belongs to space $L^\infty$, the Lagrange multipliers are finitely additive measures rather than functions. This require us to  impose some conditions  under which Lagrange multipliers are functions in $L^p$-spaces. 

By using the Yosida-Hewitt theorem, A.R\"{o}sch and F.Tr\"{o}ltzsch showed in \cite{Rosch1} and \cite{Rosch2} that, under certain conditions, the Lagrange multipliers are of $L^p$-spaces. Recently, Binh et al.  \cite{Binh} have studied regularity of the Lagrange multipliers for optimal control problems governed by semilinear elliptic equations. Their result was proven by using a fact in \cite{Ioffe-1979} on the representation of linear functionals on $L^\infty$ by a densities in $L^1$. In this paper we will prove the obtained results by  other means. Namely, we will  embed the problem \eqref{P1}-\eqref{P4} into a specific mathematical programming problem under which the Robinson constraint qualification is satisfied. Then we  can establish first and second-order optimality conditions of KKT-type. Under the Robinson constraint qualification and optimality condition in $u$, we get the regularity of multipliers directly without using the fact from \cite{Ioffe-1979}. Besides, we will show that if the initial value is good enough and boundary $\partial\Omega$ has a property of positive geometric density, then multipliers and optimal solutions are H\"{o}lder continuous. In order to prove the H\"{o}lder continuity of optimal solutions and multipliers, we will follows the approach as in the proof of \cite[Theorem 6.1]{Rosch2} with some modifications (see Remark 4.1).  

The paper is organized as follows. Section 2 is devoted to auxiliary results. The regularity of multipliers and second-order optimality conditions are provided in Section 3.  Section 4 presents a result on H\"{o}der continuity of multipliers and optimal solutions.

\section{Some auxiliary facts}

\subsection{ State equation and linearized equation}

In this section and Section 3,  we assume that $\partial \Omega$ is of class $C^2$,  $H=L^2(\Omega)$, $V=H_0^1(\Omega)$ and $D= H^2(\Omega)\cap H_0^1(\Omega)$. The norm and the scalar product in $H$ are denoted by $|\cdot|$ and $(\cdot, \cdot)$,  respectively.  Note that $|\cdot|^2=(\cdot, \cdot)$. It is known that the embeddings 
$$
D\hookrightarrow V\hookrightarrow H
$$ are compact and each space is dense in the following one. 

Given $0<\alpha\leq 1$, we denote by $C^{0, \alpha}(\bar \Omega)$ and $C^{0, \alpha}(\bar Q)$ the space of  H\"{o}lder continuous functions of order $\alpha$,  on $\bar \Omega$ and $\bar Q$, respectively. 

Let $H^{-1}(\Omega)$ be the dual of $H_0^1(\Omega)$. We define the following function spaces
\begin{align*}
& H^1(Q)=W^{1,2}(Q)=\{y\in L^2(Q): \frac{\partial y}{\partial t}, \frac{\partial y}{\partial x_i}\in L^2(Q)\},\\
&W(0, T)=\{y\in L^2(0, T; H^1_0(\Omega)): y_t\in L^2(0, T; H^{-1}(\Omega))\},\\
& W^{1,p}(0, T; H, H)=\{y\in L^p([0, T], H): \frac{\partial y}{\partial t}\in L^p([0, T], H) \},\\
& W^{1,p}(0, T; D, H)=\{y\in L^p([0, T], D): \frac{\partial y}{\partial t}\in L^p([0, T], H) \},\\
 &U=L^\infty(Q),\\ 
 &Y = \bigg\{ y \in W^{1,2}(0, T; D, H) \cap L^\infty(Q): \quad \frac{\partial y}{\partial t} + Ay \in L^p(0, T; H) \bigg\}.
\end{align*} $Y$ is endowed with the graph norm
\begin{align*}
    \|y\|_Y:= \|y\|_{W^{1,2}(0, T; D, H)} + \|y\|_{L^\infty(Q)} +  \bigg\|\frac{\partial y}{\partial t} + Ay\bigg\|_{L^p(0, T; H)}.
\end{align*} Under the graph norm $Y$ is a Banach space. Since 
\begin{align}\label{keyEmbed1}
W^{1,2}(0, T; D, H)\hookrightarrow  C(0, T; V),  
\end{align} we have 
\begin{align}\label{keyEmbed2}
Y \hookrightarrow  W^{1,2}(0, T; D, H)\hookrightarrow  C(0, T; V).   
\end{align} 

Recall that given  $y_0\in H$ and $u\in L^2(0, T; H)$,  a function $y\in W(0, T)$ is said to be a {\it weak solution} of the semilinear parabolic equation \eqref{P2}-\eqref{P3} if 
$$
\langle y_t, v\rangle + \sum_{i,j=1}^n\int_\Omega a_{ij}D_iy D_j v dx +(f(y), v) =(u, v)\quad \forall v\in H_0^1(\Omega)
$$ and a.a. $t\in [0, T]$, and $y(0)=y_0$. If a  weak solution $y$ such that $y\in W^{1,2}(0, T; D, H)$ and $f(y)\in L^2(0, T; H)$ then we have $y_t +Ay +f(y)\in L^2(0, T; H)$ and
$$
\langle y_t, v\rangle + (Ay, v) +(f(y), v) =(u, v)\quad \forall v\in H_0^1(\Omega).
$$ Since $H_0^1(\Omega)$ is dense in $H=L^2(\Omega)$, we have  
$$
(y_t, v) + (Ay, v) +(f(y), v) =(u, v)\quad \forall v\in H.
$$  Hence
$$
y_t + Ay + f(y) =u\quad \text{a.a.}\ t\in [0, T],\ y(0)=y_0.
$$ In this case we say $y$ is a {\it strong solution} of \eqref{P2}-\eqref{P3}. From now on a  solution to \eqref{P2}-\eqref{P3} is understood a strong solution. The existence of strong solutions is provided in Lemma \ref{Lemma-stateEq} below.

Let us make the following assumptions which are related to the state equation. 

\noindent $(H1)$ Coefficients ${a_{ij}}=a_{ji} \in {L^\infty }\left( \Omega  \right)$ for every $1 \le i,j \le N$, satisfy the uniform ellipticity conditions, i.e.,  there exists a number $\alpha > 0$ such that 
\begin{align}
\alpha {\left| \xi  \right|^2} \le \sum\limits_{i,j = 1}^N {{a_{ij}}\left( x \right){\xi _i}{\xi _j}} \,\,\,\,\,\,\,\,\forall \xi  \in {{\mathbb R}^N}\,\,\,{\rm {for\,\,\,a.a.}}\,\,\,x \in \Omega .
\end{align}

\noindent $(H2)$ The mapping $f:{\mathbb R} \to {\mathbb R}$ is of class $C^2$ such that $f(0)=0$ and there exists a number $C_f\in\mathbb{R}$ such that $f'(y)\geq C_f$. Furthermore,  for each $M>0$, there exists $k_M>0$ such that 
\begin{align}\label{Phi-LocalLip}
&|f'(y)| \le k_M, \\
&|f(y_1)-f(y_2)|+ |f'(y_1)-f'(y_2)|+|f''(y_1)-f''(y_2)|\leq k_M|y_1-y_2|
\end{align} for all $y, y_1, y_2$ satisfying $ |y|, |y_1|, |y_2|\leq M$.

\begin{lemma}\label{Lemma-stateEq} Suppose that $(H1)$ and $(H2)$ are satisfied and $y_0\in L^\infty(\Omega)\cap H_0^1(\Omega)$. Then for each $u\in L^p(0, T; H)$ with $p>\frac{4}{4-N}$, the state equation \eqref{P2}-\eqref{P3} has a unique  solution   $y\in Y$ and there exist positive constants $C_1>0$ and $C_2>0$ such that 
\begin{align}\label{KeyInq0}
    \|y\|_{L^\infty(Q)}+\|f(y)\|_{L^2([0, T], H)}\leq C_1 (\|u\|_{L^p([0, T], H)} + \|y_0\|_{L^\infty(\Omega)})
\end{align} and 
\begin{align}\label{KeyInq1}
    \| \frac{\partial y}{\partial t}\|_{L^2([0, T], H)} +\|y\|_{L^2([0, T], D)}\leq C_2\big(\|u\|_{L^2([0, T], H)} +\|y_0\|_{L^\infty(\Omega)} +\|y_0\|_V\big).
\end{align}  
\end{lemma}
\begin{proof} By \cite[Theorem 2.1]{Casas-2023} and \cite[Lemma 2.2.III]{Casas-2023}, for each $u\in L^p(0, T; H)$ with $p>\frac{4}{4-N}$, the state equation has a unique solution $y\in L^\infty(Q)\cap W^{1,2}(0, T; V, H)$ such that $f(y)\in L^2(0, T; H)$ and inequality \eqref{KeyInq0} is fulfilled.  By $(H2)$, we have $f(y)\in L^\infty(Q)$. We claim that $y\in W^{1,2}(0, T; D, H)$. In fact,  we have $u-f(y)\in L^2(Q)$ and  
$$
y_t + Ay = u-f(y),\  y(0)=y_0.
$$ By \cite[Theorem 5, p. 360]{Evan}, we get $y\in W^{1,2}(0, T; D, H)\cap L^\infty(0, T; V)$ and \eqref{KeyInq1} is valid. Since $u-f(y)\in L^p(0, T; H)$, we have $y_t+Ay\in L^p(0, T; H)$. The lemma is proved.

\end{proof}

Let us consider the linearized equation 
\begin{align}\label{LinearizedEq1}
   y_t + A y + c(x,t)y =u,\quad y(0)=y_0. 
\end{align} 

\begin{lemma}\label{Lemma-LinearizedEq} Suppose that $u\in L^p(0, T;  H)$ with $p>\frac{4}{4-N}$,  $y_0\in L^\infty(\Omega)\cap H^1_0(\Omega)$ and $c\in L^\infty (Q)$. Then equation \eqref{LinearizedEq1} has a unique solution $y\in Y$.
\end{lemma}
\begin{proof} Since  $c\in L^\infty(Q)$, there exists a constant $c_0>0$ such that $|c(x,t)|\leq c_0$ for a.a. $(x, t)\in Q$.  By Theorem 5 in \cite[Chapter 7]{Evan}, equation \eqref{LinearizedEq1} has a unique solution $y\in W^{1,2}(0, T; D, H)\cap L^\infty(0, T; H_0^1(\Omega))$. It follows that $c(x,t) y\in L^p(0, T; H)$. Let us consider the equation
\begin{align}\label{LinearizeEq2}
z_t + Az + c_0 z = u + (c_0-c)y,\quad z(0)=y_0. 
\end{align} By \cite[Theorem 5, Chapter 7]{Evan} and Lemma \ref{Lemma-stateEq}, the equation has a unique solution $z\in  Y$.  From \eqref{LinearizedEq1} and \eqref{LinearizeEq2}, we have
\begin{align}
z_t-y_t +A (z-y) + c_0(z-y)=0,\ (z-y)(0)=0. 
\end{align} Taking the scalar product both side with $z-y$ and using $(H1)$, we get
$$
\frac{1}2\frac{d}{dt}|z-y|^2 + \alpha |\nabla (z-y)|^2 + c_0|z-y|^2 \leq 0. 
$$    Integrating on $[0, T]$, we obtain
$$
\sup_{t\in[0, T]}|z-y|^2  + \alpha \|z-y\|^2_{L^2(0, T; D)} + c_0\|z-y\|^2_{L^2(0, T; H)}\leq 0.
$$ This implies that $z=y$ a.a. in $Q$. The conclusion is followed. 
\end{proof}

\subsection{A specific mathematical programming problem}

Let $E, W$  and $Z$ be  Banach spaces with the dual spaces $E^*$, $W^*$ and  $Z^*$, respectively. Let $M$ be a nonempty and closed subset in $Z$ and $\bar z\in M$. The set
\begin{align*}
T(M, \bar z):&=\left\{h\in Z\;|\; \exists t_n\to 0^+, \exists h_n\to h, \bar z+t_nh_n\in M \ \ \forall n\in\mathbb{N}\right\},
\end{align*}
is called  {\em the contingent cone} to $M$ at $\bar z$, respectively. The set 
$$
N(M, \bar z):=\{z^*\in Z^*| \langle z^*, z\rangle\leq 0\quad\forall z\in T(M, \bar z)\}
$$ is called {\em the normal cone} to $M$ at $\bar z$.  It is well-known that when $M$ is convex, then
$$
T (M, \bar z)=\overline{\rm cone}(M-\bar z)
$$ and 
$$
N(M, \bar z)=\{z^*| \langle z^*, z-\bar z\rangle \leq 0\quad \forall z\in M\},
$$
where 
$$
{\rm cone }(M-\bar z):=\{\lambda (h-\bar z)\;|\; h\in M, \lambda>0\}.
$$
Let $\psi: Z \to {\mathbb R}$; $F: Z \to W$ and $G: Z \to E$ be given mappings and $K$ be a nonempty closed convex subset in $E$. We now consider the optimization problem
\begin{align}
&\psi(z) \to \inf \label{mP1}\\
&\text{s.t.}\notag\\
&F(z) = 0, \label{mP2}\\
&G(z) \in K.\label{mP3}
\end{align}  Given a feasible point $z_0$ of  \eqref{mP1}-\eqref{mP3}, we impose the following assumptions.

\noindent $(A1)$ The mappings $\psi, F$ and $G$ are of class $C^2$ around $z_0$.

By define $K_1=\{0\}\times K$ and $G_1=(F, G)$. Then constraints \eqref{mP2} and \eqref{mP3} are equivalent to the constraint:
$G_1(z) \in K_1$.

Recall that a feasible point  $z_0$ is said to satisfy the Robinson constraint qualification  if
\begin{align}\label{RobinsonCond1}
0\in {\rm int} \{G_1(z_0) + DG_1(z_0)Z -K_1\}. 
\end{align} By Proposition 2.95 in \cite{Bonnans-2000}, condition \eqref{RobinsonCond1} is equivalent to
\begin{align}\label{RobinsonCond2}
   DG_1(z_0)Z -{\rm cone}(K_1-G_1(z_0))=W\times E.  
\end{align} This means that for any $(v, e)\in W\times E$, there exists $z\in Z$ and $w\in {\rm cone}(K-G(z_0))$ such that 
\begin{equation}
\begin{cases}
DF(z_0)z=v,\\
DG(z_0)z-w=e.
\end{cases}
\end{equation}  The Robinson constraint qualification guarantees that the necessary optimality conditions are of KKT type. In this section we need this condition. 

\noindent $(A2)$ $z_0$ satisfies the Robinson constraint qualification \eqref{RobinsonCond2}.

\medskip

Let us denote by ${\mathcal C}_0[z_0]$  the set of vectors $d \in Z$ satisfying the following conditions:
\begin{itemize}
\item [$(a_1)$] $D\psi(z_0)d \leq 0,$

\item [$(a_2)$] $DF(z_0)d = 0,$

\item [$(a_3)$] $DG(z_0)d  \in {\rm cone}(K-G(z_0))$
\end{itemize}  and by ${\mathcal C}[z_0]$ the closure of ${\mathcal C}_0[z_0]$ in $Z$, which is called the critical cone at $z_0$. Each vector $d\in \mathcal{C}[z_0]$ is called a critical direction.  Let 
\begin{align}
{\mathcal L}\left( {z,{v^*},{e^*}} \right): = \psi(z) + \langle v^*, F(z)\rangle +\langle e^*, G(z)\rangle 
\end{align} be the Lagrange function  associated with the problem \eqref{mP1}--\eqref{mP3}, where $z \in Z $, $v^* \in V^*$ and $e^* \in E^*$. We say that $v^*$ and $e^*$ are Lagrange multipliers at $z_0$ if the following conditions are fulfilled:
\begin{align}
&D_z {\mathcal L}\left( {z_0,{v^*},{e^*}} \right)= D\psi(z_0) + DF(z_0)^*v^* + DG(z_0)^* e^*=0,\\
& e^*\in N(K, G(z_0)). 
\end{align} We denote by $\Lambda[z_0]$ the set of Lagrange multipliers at $z_0$. 

We have the following result on  first and second-order necessary optimality conditions. 

\begin{lemma}\label{Lemma-Optim} Suppose that $(A1)$ and $(A2)$ are satisfied, and $z_0$ is a locally optimal solution. Then $\Lambda[z_0]$ is nonempty and 
\begin{align}\label{NonnegativeSOC}
\sup_{(v^*, e^*)\in\Lambda[z_0]}\big(D^2\psi(z_0) d^2 + v^* D^2F(z_0)d^2 + e^*D^2G(z_0)d^2\big)\geq 0\quad\forall d\in{\mathcal C}[z_0]. 
\end{align} 
\end{lemma}
\begin{proof} By \cite[Theorem 3.9]{Bonnans-2000}, the set $\Lambda[z_0]$ is nonempty and weakly star compact. By \cite[Theorem 3.45]{Bonnans-2000}, \eqref{NonnegativeSOC} is valid for all $d\in \mathcal{C}_0[z_0]$. Note that for each $d\in\mathcal{C}_0[z_0]$, the sigma term equal to zero. By the compactness of $\Lambda[z_0]$, the mapping $d\mapsto \sup_{(v^*, e^*)\in\Lambda[z_0]} D^2_z\mathcal{L}(\bar z, v^*, e^*)[d,d]$ is continuous. Therefore, inequality \eqref{NonnegativeSOC} is valid for $d\in\mathcal{C}[z_0]$ which is the closure of $\mathcal{C}_0[z_0]$. 
\end{proof}

\section{Existence of regular multipliers in  second-order optimality conditions}

Let us denote by $\Phi$ the feasible set of the problem \eqref{P1}-\eqref{P4}. Hereafter, the symbol $\phi[x,t]$ stands for $\phi(x, t, \bar y(x, t), \bar u(x, t))$ with $(\bar y, \bar u)\in\Phi$.

Let $\phi: \Omega\times[0, T]\times\mathbb{R}\times\mathbb{R}\to\mathbb{R}$ be a mapping which stands for $L$ and $g$. We impose the following hypotheses. 

\noindent $(H3)$ $\phi$ is a Carath\'{e}odory  function and  for each $(x,t)\in \Omega\times [0, T]$, $\phi(x, t, \cdot, \cdot)$ is of class $C^2$ and satisfies the following property:  for each $M>0$, there exists $l_{\phi,M}>0$ such that 
\begin{align*}
    &\quad |\phi(x, t, y_1, u_1)-\phi(x, t, y_2, u_2)|+|\phi_y(x, t, y_1, u_1)-\phi_y(x, t, y_2, u_2)|\\
    &+|\phi_u(x, t, y_1, u_1)-\phi_u(x, t, y_2, u_2)|+|\phi_{yy}(x, t, y_1, u_1)-\phi_{yy}(x, t, y_2, u_2)|\\
    &+|\phi_{yu}(x, t, y_1, u_1)-\phi_{yu}(x, t, y_2, u_2)|+ |\phi_{uu}(x, t, y_1, u_1)-\phi_{uu}(x, t, y_2, u_2)|\\
    &\leq l_{\phi,M}(|y_1-y_2|+ |u_1-u_2|)
\end{align*}  for all $(x, t, y_i, u_i)\in \Omega\times[0, T]\times \mathbb{R}\times\mathbb{R}$ satisfying $|y_i|, |u_i|\leq M$ with $i=1,2$. Furthermore, we require that the functions $\phi_y(\cdot, \cdot, 0, 0), \phi_u(\cdot, \cdot, 0, 0), \phi_{yy}(\cdot, \cdot, 0, 0), \phi_{yu}(\cdot, \cdot, 0, 0)$ and $\phi_{uu}(\cdot, \cdot, 0, 0)$ belong to $L^\infty(Q)$. 

\noindent $(H4)$ The function $\frac{1}{g_u[\cdot, \cdot]}$ belongs to $L^\infty (Q)$.

Note that hypothesis $(H3)$ makes sure that $J$ and $g$ are of class $C^2$ on $Y\times U$. Meanwhile, $(H4)$ guarantees that the Robinson constraint qualification is satisfied and the Lagrange multipliers belong to $L^1(\Omega)$.  Let us give an illustrative example showing that $g$ satisfies assumption $(H4)$.  

\begin{example}{\rm 
The following formulae of $g$ satisfies assumption $(H4)$:
\begin{align*}
    & g(x, t, y, u) = u, \\
    & g(x, t, y, u) = \frac{1}{3}y^2u^3 - yu^2 + (2 + |t-1|)u, \\
    & g(x, t, y, u) = \pi(x, t) + y^4u^3 + (y^2 + 1)u
\end{align*}
where $\pi$ is a continuous function on $Q$. 
}
\end{example}

In the sequel, we assume that 
\begin{align}
K_\infty=\{u\in U: u(x, t)\leq 0\ {\rm a.a.}\ (x, t)\in Q\}.  
\end{align}
Let ${\mathcal C}_0 (\bar z)$ be a set of all $z = (y, u) \in Y\times U$ such that the following conditions hold: 
\begin{itemize}
\item [$(c_1)$] $\displaystyle \int_Q (L_y[x, t]y(x, t) + L_u[x, t]u(x, t))dxdt  \le 0$; 

\item [$(c_2)$] $\dfrac{\partial y}{\partial t} + Ay + f'(\bar y)y = u, \quad y(0) = 0;$ 

\item[$(c_3)$] $g_y[\cdot, \cdot]y + g_u[\cdot, \cdot]u \in {\rm cone}(K_\infty - g[\cdot, \cdot]).$ 
\end{itemize}
Define ${\mathcal C}_2^*(\bar z) =  \bar {{\mathcal C}_0 (\bar z)}$, which is the closure of ${\mathcal C}_0 (\bar z)$ in $W^{1,2}(0, T; D, H)\times L^2(Q)$. The set ${\mathcal C}_2^*(\bar z)$ is called a critical cone to problem \eqref{P1}-\eqref{P4} at $\bar z$. Each vector $(y, u)\in\mathcal{C}_2^*[\bar z]$ is called a critical direction. 

The following theorem gives  first and second-order necessary optimality conditions and the regularity of multipliers.

\begin{theorem}\label{Theorem 1}  Suppose that assumptions $(H1)-(H4)$ are satisfied and $(\bar y, \bar u) \in \Phi$ is a locally optimal solution of the problem \eqref{P1}-\eqref{P4}.  Then  there exist a function $\varphi \in W^{1, 2}(0, T; D, H) \cap L^\infty(Q)$ and a function $e\in L^\infty(Q)$ such that the following conditions are fulfilled:

\noindent $(i)$ (the adjoint equation) 
$$
-\dfrac{\partial \varphi}{\partial t} + A^* \varphi + f'(\bar y)\varphi = -L_y[.,.] - eg_y[.,.], \quad \varphi(., T) = 0, 
$$ where $A^*$ is the adjoint operator of $A$, which is defined by 
$$
A^*\varphi =  - \sum_{i,j = 1}^N D_i\left(a_{ij}\left( x \right) D_j \varphi \right);
$$ 
 
 \noindent $(ii)$ (the stationary condition in $u$) 
$$
L_u[x, t] - \varphi(x,t) + e(x,t)g_u[x, t] = 0 \ \ {\rm a.a.} \  (x, t) \in Q;
$$

\noindent $(iii)$ (the complementary condition) 
\begin{align}\label{ComlementCond}
e(x,t)g[x, t]=0 \   {\rm and} \   e(x, t)\geq 0\quad {\rm a.a.} \ (x,t)\in Q;
\end{align}

\noindent $(iv)$ (the nonnegative second-order condition)
    \begin{align}\label{SOC1}
        &\int_Q (L_{yy}[x, t]y^2 + 2L_{yu}[x, t]yu + L_{uu}u^2)dxdt \nonumber \\
        &+ \int_Q e(g_{yy}[x, t]y^2 + 2g_{yu}[x, t]yu + g_{uu}u^2)dxdt + \int_Q \varphi f''(\bar y)y^2dxdt \geq 0\quad \forall (y, u) \in \mathcal C_2^*[(\bar y, \bar u)].
    \end{align} 
\end{theorem}
\begin{proof}  Let us define Banach spaces
\begin{align*}
    &Z = Y \times U,\ W = W_1 \times W_2,\quad E = L^\infty(Q), \\
    &W_1 = L^p(0, T; H),\  W_2 =  L^\infty(\Omega) \cap H_0^1(\Omega) \quad {\rm with} \quad    p > \frac{4}{4 - N}. 
\end{align*}
Define mappings $F: Y\times U\to W$ and  $G: Y\times U\to E$ by setting
\begin{align}
     F(y, u) = (F_1(y, u),\   F_2(y, u)) = \bigg(\frac{\partial y}{\partial t} + Ay + f(y) - u,\ \   y(0) - y_0 \bigg)
\end{align} and 
\begin{align}   
     G(y, u) = g(\cdot, \cdot, y,  u).
\end{align}
By definition of space $Y$, if $(y, u) \in Y\times U$ then $\frac{\partial y}{\partial t} + Ay \in L^p(0, T; H),$ $f(y) \in L^\infty(Q) \subset L^p(0, T; H)$ (since $y \in L^\infty(Q)$) and $u \in L^\infty(Q) \subset L^p(0, T; H)$. Hence  
$$
\frac{\partial y}{\partial t} + Ay + f(y) - u \in L^p(0, T; H) = W_1
$$ and $F_1$ is well defined.  Take any  $y \in Y$. By \eqref{keyEmbed2}, we have $y\in C([0, T], V)$. Hence $y(0)\in H_0^1(\Omega)$. Since $y\in L^\infty(Q)$, there exists a constant $M>0$ such that $|y(t, x)|\leq M$ for a.a. $(t, x)\in Q$. Particularly, if $t_n\in [0, T]$ such that $t_n\to 0$, we have 
$|y(t_n, x)|\leq M$ for all $n$ and a.a. $x\in\Omega$ and 
$$
y(0, \cdot)=\lim_{n\to\infty} y(t_n, \cdot)\quad \text{in}\quad V.
$$ This implies that $\|y(0, \cdot)-y(t_n, \cdot)\|_{L^2(\Omega)}\to 0$ as $n\to \infty$. By passing to a subsequence, we may assume that $y(t_n, x)\to y(0, x)$ for a.a. $x\in\Omega$. It follows that 
\begin{align*}
 |y(0, x)|\leq |y(0, x)-y(t_n, x)| + |y(t_n, x)|\leq |y(0, x)-y(t_n, x)| + M.
\end{align*} By passing to the limit when $n\to \infty$, we get $|y(0, x)|\leq M$. Hence $y(0, \cdot)\in L^\infty(Q)$. Consequently, $y\in H_0^1(\Omega)\cap L^\infty(Q)$. Hence  $F_2$ is well defined and so is $F$. 

We now formulate problem \eqref{P1}-\eqref{P4} in the form
\begin{align}
    &J(y,u) \to {\rm min}, \\
    &{\rm s.t.} \nonumber \\
    &F(y,u) = 0,\\
    &G(y,u) \in K_\infty.
    \end{align}
We are going to apply Lemma \ref{Lemma-Optim}  for the problem \eqref{P1}-\eqref{P4} in order to derive necessary optimality conditions. 

\noindent \textbf{Step 1.} \textit{Verification of assumptions $(A1)$ and $(A2)$.} 

\noindent $\bullet$ Verification of $(A1)$. From $(H2)$ and $(H3)$ we see that the mappings $J, F$ and $G$ are of class $C^2$ around $\bar z  = (\bar y, \bar u)$. Hence, $(A1)$ is valid. Here $DJ(\bar z),  D^2J(\bar z), DF(\bar z),  D^2F(\bar z)$,  $DG(\bar z)$ and $D^2G(\bar z)$ are defined by 
\begin{align}
    &DJ(\bar z)(y, u) = \int_Q(L_y[x, t]y + L_u[x, t]u)dxdt, \\
    &D^2J(\bar z)(y, u)^2 = \int_Q(L_{yy}[x, t]y^2 + 2L_{yu}[x, t]yu + L_{uu}[x, t]u^2)dxdt, \label{A1.1} \\
    &DF(\bar z)(y, u) = (DF_1(\bar z)(y, u), DF_2(\bar z)(y, u)) \\
    & \ \ \ {\rm where} \ \ \ DF_1(\bar z)(y, u) = \dfrac{\partial y}{\partial t} + Ay + f'(\bar y)y - u \ \ \ {\rm and} \ \ \ DF_2(\bar z)(y, u) = y(0), \\
    &D^2F(\bar z)(y, u)^2 = \bigg(D^2F_1(\bar z)(y, u)^2, \ \ D^2F_2(\bar z)(y, u)^2\bigg) \label{A1.2} \\
    & \ \ \ {\rm where} \ \ \ D^2F_1(\bar z)(y, u)^2 = f''(\bar y)y^2 \ \ \ {\rm and} \ \ \ D^2F_2(\bar z)(y, u)^2 = 0,\\
    &DG(\bar z)(y, u) = g_y[x, t]y + g_u[x, t]u, \\
    &D^2G(\bar z)(y, u)^2 = g_{yy}[x, t]y^2 + 2g_{yu}[x, t]yu + g_{uu}[x, t]u^2,  \label{A1.3}
\end{align}
for all $(y, u) \in Z$.

\noindent $\bullet$ Verification of $(A2)$. For this we show that the Robinson constraint qualification \eqref{RobinsonCond2} is valid.  Take any $(\hat u, \hat y_0)\in W$ and $e\in E$, we want to show that there exists $(y, u)\in Y\times U$ and $w\in {\rm cone}(K-g[\cdot, \cdot])$ such that
\begin{align}\label{RobinEq0}
\begin{cases}
y_t + Ay + f'(\bar y) y - u= \hat u, \quad  y(0)=\hat y_0,\\
g_y[x,t] y + g_u[x,t]u - w= e.
\end{cases}
\end{align} For this we choose $w=0$ and consider the equation
\begin{align}
\begin{cases}
y_t +A y + f'(\bar y) y -\frac{e-g_y[\cdot, \cdot]y}{g_u[\cdot, \cdot]}=\hat u\\
y(0)=\hat y_0.
\end{cases}
\end{align} This is equivalent to the equation
\begin{align}\label{RobinsonEq}
\begin{cases}
y_t +A y + (f'(\bar y)+\frac{g_y[\cdot, \cdot]}{g_u[\cdot, \cdot]})y=\hat u +\frac{e}{g_u[\cdot, \cdot]}\\
y(0)=\hat y_0.
\end{cases}
\end{align} By $(H4)$, we have $\frac{e}{g_u[\cdot, \cdot]}\in L^\infty(Q)$, $f'(\bar y)+\frac{g_y[\cdot, \cdot]}{g_u[\cdot, \cdot]}\in L^\infty(Q)$ and  $\hat u+ \frac{e}{g_u[\cdot, \cdot]}\in L^p(0, T; H)$.  According to Lemma \ref{Lemma-LinearizedEq}, equation \eqref{RobinsonEq} has a unique solution $y\in Y$. 
We now define 
$$ 
u= \dfrac{e-g_y[\cdot, \cdot]y}{ g_u[\cdot, \cdot]}.
$$ Then we have $u \in U$, $g_y[\cdot, \cdot]y + g_u[\cdot, \cdot] u - 0 = e$ and  
\begin{align*}
y_t + Ay + f'(\bar y) y -u =\hat u,\quad y(0)=\hat y_0. 
\end{align*} Hence \eqref{RobinEq0} is fulfilled and the Robison constraint qualification is valid. 

\noindent \textbf{Step 2.} \textit{Deriving optimality conditions.} Let $q$ be a conjugate number of $p$ with $p > \frac{4}{4 - N}$, that is, $\frac{1}{p} + \frac{1}{q} = 1$. Since $H = L^2(\Omega)$, $H=H^*$. Hence dual of $L^p(0, T; H)$ is $L^q(0, T; H)$. We recall that $E^* = (L^\infty(Q))^*$ and $W^* = W_1^* \times W_2^*$, where $W_1^* = L^p(0, T; H)^* = L^q(0, T; H) \subset L^1(0, T; H)$ and $W_2^* = \big( L^\infty(\Omega) \cap H_0^1(\Omega) \big)^*$. In this case the Lagrange function which is associated with the problem is defined as follows
\begin{align*}
    &\mathcal{L}: Y \times U \times L^q(0, T; H) \times \big( L^\infty(\Omega) \cap H_0^1(\Omega) \big)^* \times (L^\infty(Q))^* \to  \mathbb R, \\
    &\mathcal{L}(y, u, \phi_1, \phi_2^*, e^*) = \int_QL(x, t, y, u)dxdt + \int_Q\bigg[\phi_1 (\frac{\partial y}{\partial t} + Ay + f(y) -u) \bigg]dxdt \\
    &\quad \quad \quad\quad \quad \quad \quad \quad + \left\langle \phi_2^*, y(0) - y_0 \right\rangle_{W_2^*, W_2} + \left\langle e^*, G(y, u)\right\rangle_{L^\infty(Q)^*, L^\infty(Q)}.
\end{align*}
As it was shown in Step 1, all assumptions of Lemma \ref{Lemma-Optim} are fulfilled. Therefore,  there exist $\phi_1 \in L^q(0, T; H)$, $\phi_2^* \in \big( L^\infty(\Omega) \cap H_0^1(\Omega) \big)^*$ and $e^* \in L^\infty(Q)^*$ such that the following conditions hold:
\begin{align}
    &e^* \in N(K_\infty, g[\cdot, \cdot]),  \label{dk1} \\
    &D_{(y, u)}{\mathcal L}(\bar z, \phi_1, \phi_2^*, e^*) = 0.  \label{dk2}
    \end{align}
The condition (\ref{dk2}) is equivalent to the two following relations 
\begin{align}\label{dk2.1}
    \int_QL_y[x, t]ydxdt &+ \int_Q\bigg[\phi_1 (\frac{\partial y}{\partial t} + Ay + f'(\bar y)y  \bigg]dxdt \nonumber \\
    &+ \left\langle \phi_2^*, y(0) \right\rangle_{W_2^*, W_2} + \left\langle e^*, g_y[x, t]y\right\rangle_{L^\infty(Q)^*, L^\infty(Q)} = 0, \quad \quad {\rm {for \ all}} \quad y \in Y
\end{align}
and
\begin{align}\label{dk2.2}
    \int_Q L_u[x, t]udxdt - \int_Q\phi_1udxdt + \left\langle e^*, g_u[x, t]u\right\rangle_{L^\infty(Q)^*, L^\infty(Q)} = 0, \quad \quad {\rm {for \ all}} \quad u \in U.
\end{align}
For each $y \in Y$, take $u: = \frac{g_y[\cdot, \cdot]}{g_u[\cdot, \cdot]}y$. Since $g_y[\cdot, \cdot] \in L^\infty(Q)$, $\frac{1}{g_u[\cdot, \cdot]}\in L^\infty (Q)$ and $y \in L^\infty(Q)$, we get $u \in L^\infty(Q) = U$. Then, we have
\begin{align*}
    &\left\langle e^*, g_u[x, t]u\right\rangle_{L^\infty(Q)^*, L^\infty(Q)} = \left\langle e^*, g_u[x, t]\frac{g_y[x, t]}{g_u[x, t]}y\right\rangle_{L^\infty(Q)^*, L^\infty(Q)} = \left\langle e^*, g_y[x, t]y\right\rangle_{L^\infty(Q)^*, L^\infty(Q)},
\end{align*}
and insert  $u = \frac{g_y[\cdot, \cdot]}{g_u[\cdot, \cdot]}y$ into (\ref{dk2.2}), we obtain 
\begin{align*}
    \left\langle e^*, g_y[x, t]y\right\rangle_{L^\infty(Q)^*, L^\infty(Q)} = - \int_Q L_u[x, t]\frac{g_y[x, t]}{g_u[x, t]}ydxdt + \int_Q\phi_1\frac{g_y[x, t]}{g_u[x, t]}ydxdt
\end{align*}
Combining this with (\ref{dk2.1}), we get
\begin{align}\label{dk2.1-MODIFY}
\int_Q\bigg( \frac{\partial y}{\partial t}  + Ay  + h(y) \bigg)\phi_1 dxdt  &+ \left\langle \phi_2^*, y(0) \right\rangle_{W_2^*, W_2}  \nonumber \\ 
&= \int_Q\bigg(-L_y[x, t] + \frac{g_y[x, t]}{g_u[x, t]}L_u[x, t] \bigg)ydxdt\ \ \quad  \forall y\in Y,  
\end{align}
 where the function $h: \Omega \times [0, T]\times\mathbb R \to \mathbb R$ is defined by setting 
\begin{align}\label{h-function}
h(x, t, y) =\bigg(f'(\bar y(x, t)) +\frac{g_y[x, t]}{g_u[x, t]}\bigg)y. 
\end{align}
Let us claim that $e^*$ can be represented by a density in $L^q(0, T; H)$.  In fact, from \eqref{dk2.2}, we have 
\begin{align}\label{StationaryCondU}
g_u[\cdot, \cdot]e^*=- L_u[\cdot, \cdot] +\phi_1.
\end{align}
 Hence 
$$
e^*=\frac{\phi_1-L_u[\cdot, \cdot]}{g_u[\cdot, \cdot]}\quad {\rm on}\quad L^\infty(Q).
$$  
From $(H4)$ and $\phi_1 \in L^q(0, T; H)$, $L_u[\cdot, \cdot] \in L^\infty(Q)$, we get $\frac{\phi_1-L_u[\cdot, \cdot]}{g_u[\cdot, \cdot]}\in L^q(0, T; H)$. Hence we can identify $e^*$ with a function $e\in L^q(0, T; H)$. The claim is justified.  Then from \eqref{dk1}, we have 
$$
\int_Q e(x,t)(v- g[x, t]) dxdt\leq 0\quad \forall v\in K_\infty. 
$$ By \cite[Corollary 4]{Pales}, we have 
$$
e(x,t)\in N((-\infty, 0]; g[x,t])\quad \text{a.a.}\ (x,t)\in Q. 
$$ This implies $e(x, t)g[x,t]=0$ and $e(x,t)\geq 0$ a.a. $(x, t)\in Q$. We obtain assertion $(iii)$ of the theorem. 

We now rewrite \eqref{dk2.1} and \eqref{dk2.2} in  the following equivalent forms:
\begin{align}
    \label{dk2.1.a}
    \int_QL_y[x, t]ydxdt + &\int_Q\bigg[\phi_1 (\frac{\partial y}{\partial t} + Ay + f'(\bar y)y)  \bigg]dxdt \nonumber \\
    &+ \langle \phi_2^*, y(0) \rangle_{W_2^*, W_2} + \int_Q eg_y[x, t]ydxdt = 0, \quad \quad {\rm {for \ all}} \quad y \in Y.
\end{align}
Next we derive the assertion $(ii)$. For this, we consider the following parabolic equation
\begin{align}
\label{Eq.Varphi.1}
    - \frac{\partial \varphi}{\partial t}  + A^* \varphi + h(\varphi)  = -L_y[x, t] + \frac{g_y[x, t]}{g_u[x, t]}L_u[x, t], \quad \quad \varphi(., T) = 0.
\end{align}
where the function $h$ is  defined by \eqref{h-function}. By changing variable $\widetilde \varphi(t)  = \varphi(T-t)$, the equation (\ref{Eq.Varphi.1}) becomes
\begin{align}
    \label{Eq.Varphi.2}
    &\frac{\partial \widetilde \varphi (t)}{\partial t} + A^* \widetilde \varphi (t) + \bigg[ f'(\bar y (x, T - t)) + \frac{g_y[x, T-t]}{g_u[x, T-t]}  \bigg] \widetilde \varphi (t) = -L_y[x, T-t] + \frac{g_y[x, T-t]}{g_u[x, T-t]}L_u[x, T-t], \nonumber \\ 
    &\widetilde \varphi(., 0) = 0.
\end{align}
It is easy to show that  $L_y[\cdot, \cdot ], L_u[\cdot, \cdot ], g_y[\cdot, \cdot ] \in L^\infty(Q)$. By $(H_4)$, we have that $-L_y[., .] + \frac{g_y[., .]}{g_u[., .]}L_u[., .] \in L^\infty(Q) \subset L^p(0, T; H)$ with  $p > \frac{4}{4 - N}$.  By Lemma \ref{Lemma-LinearizedEq},  equation (\ref{Eq.Varphi.2}) has a unique solution $\widetilde \varphi \in  W^{1, 2}(0, T; D, H) \cap L^\infty(Q)$. Hence $\varphi(t) = \widetilde \varphi(T - t)$ is the unique solution to the equation (\ref{Eq.Varphi.1}) in $W^{1, 2}(0, T; D, H) \cap L^\infty(Q)$.

We now multiply both sides of equation (\ref{Eq.Varphi.1}) by $y \in Y$ and integrate on $Q$. Then we have
\begin{align}
    \label{Eq.Varphi.3}
    \int_Q\bigg( - \frac{\partial \varphi}{\partial t}  + A^* \varphi + h(\varphi) \bigg)ydxdt = \int_Q\bigg(-L_y[x, t] + \frac{g_y[x, t]}{g_u[x, t]}L_u[x, t] \bigg)ydxdt \quad  {\rm for \ all } \ \ y \in Y.
\end{align}
Besides, we have
\begin{align*}
    &\int_Q (A^*\varphi) ydxdt = \int_0^T \left \langle A^* \varphi, y(t) \right \rangle_{H, H} dt = \int_0^T \left \langle Ay, \varphi \right \rangle_{H, H} dt = \int_Q (Ay) \varphi dxdt, \\
    &\int_Q h(\varphi)y dxdt = \int_Q h(y) \varphi dxdt.
\end{align*}
Using the fact $\varphi (x, T) = 0$ on $\Omega$ and  integrating by parts, we obtain
\begin{align*}
    - \int_Q \frac{\partial \varphi}{\partial t} y dxdt = - \int_0^T \left \langle \frac{\partial \varphi}{\partial t},  y \right \rangle_{H, H} dt = \int_\Omega \varphi (0)y(0) dx + \int_Q \frac{\partial y}{\partial t} \varphi dxdt.
\end{align*}
Hence, the equation (\ref{Eq.Varphi.3}) becomes
\begin{align}
    \label{Eq.Varphi.4}
    \int_Q\bigg( \frac{\partial y}{\partial t}  + Ay  + h(y) \bigg)\varphi dxdt + \int_\Omega \varphi (0)y(0) dx = \int_Q\bigg(-L_y[x, t] + \frac{g_y[x, t]}{g_u[x, t]}L_u[x, t] \bigg)ydxdt, 
\end{align}
for all $y \in Y$. Subtracting (\ref{dk2.1-MODIFY}) from (\ref{Eq.Varphi.4}), we get 
\begin{align}
\label{Eq.Varphi.5}
    \int_Q (\varphi - \phi_1) \bigg( \frac{\partial y}{\partial t}  + Ay  + h(y) \bigg) dxdt  +  \bigg(\int_\Omega \varphi (0)y(0) dx  - \left\langle \phi_2^*, y(0) \right\rangle_{W_2^*, W_2}\bigg) = 0, \quad {\rm for \ all} \ y \in Y.
\end{align}
On the other hand, as it was shown that for every $\vartheta \in L^p(0, T; H)$ and $\psi \in  L^\infty(\Omega) \cap H_0^1(\Omega)$, the parabolic equation
\begin{align*}
    \frac{\partial y}{\partial t}  + Ay  + h(y)  = \vartheta, \quad \quad y(0) = \psi 
\end{align*}
has a unique solution $y \in Y$. From this and \eqref{Eq.Varphi.5}, we deduce that 
\begin{align*}
    &\int_Q (\varphi - \phi_1) \vartheta dxdt  =  0 \quad \quad {\rm for \ all} \  \vartheta \in L^p(0, T; H),  \\
    &\left\langle \phi_2^*, \psi \right\rangle_{W_2^*, W_2} = \int_\Omega \varphi (0)\psi dx  \quad \quad {\rm for \ all} \   \psi \in  L^\infty(\Omega) \cap H_0^1(\Omega).
\end{align*} It follow that 
\begin{align*}
    \varphi = \phi_1\  {\rm in}\ Q, \quad   \phi_2^*=\varphi (0)\   {\rm in}\ \Omega.
\end{align*}
Consequently, \eqref{StationaryCondU} becomes
$$
L_u[\cdot, \cdot]-\varphi + e g_u[\cdot, \cdot]=0.
$$ Since $L_u[\cdot, \cdot], \varphi \in L^\infty(Q)$ and $\frac{1}{g_u[\cdot, \cdot]}  \in L^\infty(Q)$, we see that 
\begin{align} \label{Eq.e.1}
    e = \frac{1}{g_u[., .]}\bigg(\varphi - L_u[., .] \bigg) \in L^\infty(Q).
\end{align}  Therefore,  assertion $(ii)$ is proved. Inserting \eqref{Eq.e.1} into  equation \eqref{Eq.Varphi.1}, we obtain the assertion $(i)$ of the theorem.  Note that the Lagrange function now becomes
\begin{align}\label{LagrangeFunction2}
& \mathcal{L}(y, u, \varphi, e)  \nonumber \\
 &=J(y, u) +\int_Q \big(y(-\varphi_t +A^*\varphi) +(f(y)-u)\varphi + eg(x,t, y, u)\big)dxdt  -  \int_\Omega \varphi(0) y_0 dx.   
\end{align}
Let us show that  $(\varphi, e)$ is unique. In fact, assume that there are two couples $(\varphi_1, e_1)$ and $(\varphi_2, e_2)$ satisfying assertions $(i)-(iii)$. Then we have 
$$
-\frac{\partial \varphi_i}{\partial t} + A^* \varphi_i + f'(\bar y)\varphi_i=-L_y[\cdot, \cdot]-e_i g_y[\cdot, \cdot],\ \varphi_i(T)=0
$$ and 
$$
L_u[\cdot, \cdot] -\varphi_i + g_u[\cdot, \cdot]e_i =0 \quad {\rm with}\ i=1,2.
$$  By subtracting equations, we get
$$
-\frac{\partial }{\partial t}(\varphi_1-\varphi_2) + A^*(\varphi_1-\varphi_2) + f'(\bar y)(\varphi_1-\varphi_2)=-(e_1-e_2) g_y[\cdot, \cdot],\ (\varphi_1-\varphi_2)(T)=0
$$ and 
$$
-(\varphi_1-\varphi_2) + g_u[\cdot, \cdot](e_1-e_2)=0.
$$ It follows that 
$$
e_1-e_2=\frac{\varphi_1-\varphi_2}{g_u[\cdot, \cdot]}
$$ and 
$$
-\frac{\partial }{\partial t}(\varphi_1-\varphi_2) + A^*(\varphi_1-\varphi_2) + (f'(\bar y)+\frac{g_y[\cdot, \cdot]}{g_u[\cdot, \cdot]})(\varphi_1-\varphi_2)=0,\ (\varphi_1-\varphi_2)(T)=0.
$$  By defining $\tilde\varphi(t)=(\varphi_1-\varphi_2)(T-t)$, we have 
$$
\frac{\partial }{\partial t} \tilde\varphi + A^*\tilde\varphi + (f'(\bar y(\cdot, T-t))+\frac{g_y[\cdot, T-t]}{g_u[\cdot, T-t]})\tilde\varphi = 0,\quad \tilde\varphi(0)=0. 
$$ Taking the scalar product both sides with $\tilde\varphi$ and using $(H1)$, we get
$$
\frac{1}{2}\frac{\partial}{\partial t}|\tilde\varphi|^2 +\alpha|\nabla \tilde\varphi|^2 +\int_\Omega (f'(\bar y(\cdot, T-t))+\frac{g_y[\cdot, T-t]}{g_u[\cdot, T-t]})\tilde\varphi^2 dx \leq 0.
$$  Since $f'(\bar y) +\frac{g_y[\cdot, \cdot]}{g_u[\cdot, \cdot]}\in L^\infty(Q)$, there exists a positive number $\varrho>0$ such that 
$$
|f'(\bar y(x,t)) +\frac{g_y[x, t]}{g_u[x, t]}|\leq \varrho\quad {\rm a.a.}\ (x, t)\in Q.
$$ It follows that 
$$
\frac{1}{2}\frac{\partial}{\partial t}|\tilde\varphi|^2 +\alpha|\nabla \tilde\varphi|^2 -\varrho|\tilde\varphi|^2 \leq 0.
$$ This implies that 
$$
\frac{1}{2}\frac{\partial}{\partial t}|\tilde\varphi|^2 \leq \varrho|\tilde\varphi|^2. 
$$ The Gronwall inequality implies that $|\tilde\varphi|^2=0$. Hence  $\varphi_1=\varphi_2$ and $e_1=e_2$. Thus $(\varphi, e)$ is unique. By Lemma \ref{Lemma-Optim}, we have 
\begin{align*}
        &D_{(y, u)}^2\mathcal{L}(\bar y, \bar u, \varphi, e)[(y, v), (y, v)]=\int_Q (L_{yy}[x, t]y^2 + 2L_{yu}[x, t]yv + L_{uu}[x,t]v^2)dxdt \\
        &+ \int_Q e(g_{yy}[x, t]y^2 + 2g_{yu}[x, t]yv + g_{uu}[x, t]v^2)dxdt + \int_Q \varphi f''(\bar y)y^2dxdt \geq 0\quad \forall (y, v) \in \mathcal C_0[(\bar y, \bar u)].
    \end{align*} Let us take any $(y, v)\in C_2^*[(\bar y, \bar u)]$. Then there exists a sequence $(y_k, v_k)\in \mathcal{C}_0[(\bar y, \bar u)]$ such that $(y_k, v_k)\to (y, v)$ in norm of  $W^{1,2}(0, T; D, H)\times L^2(Q)$. By the above, we have 
\begin{align}
    &\int_Q (L_{yy}[x, t]y_k^2 + 2L_{yu}[x, t]y_kv_k + L_{uu}[x,t]v^2)dxdt \nonumber \\
    &+ \int_Q e(g_{yy}[x, t]y_k^2 + 2g_{yu}[x, t]y_kv_k + g_{uu}[x,t]v_k^2)dxdt + \int_Q \varphi f''(\bar y)y_k^2dxdt \geq 0. \label{NonnegativeCond2}
\end{align} Since $y_k\to y$ strongly in $W^{1,2}(0, T; D, H)$, we have $y_k\to y$ strongly in $L^2(0, T; H)=L^2(Q).$ Letting $k\to\infty$, we obtain from \eqref{NonnegativeCond2} that 
\begin{align*}
        &D_{(y, u)}^2\mathcal{L}(\bar y, \bar u, \varphi, e)[(y, v), (y, v)]=\int_Q (L_{yy}[x, t]y^2 + 2L_{yu}[x, t]yv + L_{uu}[x,t]v^2)dxdt \nonumber \\
        &+ \int_Q e(g_{yy}[x, t]y^2 + 2g_{yu}[x, t]yv + g_{uu}[x,t]v^2)dxdt + \int_Q \varphi f''(\bar y)y^2dxdt \geq 0.
    \end{align*} Hence The assertion $(iv)$ is established and the proof of the theorem is complete. 
\end{proof}

To deal with second-order sufficient optimality conditions, we need to enlarge the critical cone $\mathcal{C}_2^*[\bar z]$ by the cone $\mathcal{C}_2[\bar z]$ which consists  of couples $(y, u) \in W^{1, 2}(0, T; D, H) \times L^2(Q)$ satisfying the following conditions:
\begin{itemize}
\item [$(c'_1)$] $\displaystyle \int_Q (L_y[x, t]y(x, t) + L_u[x, t]u(x, t))dxdt \le 0$; 

\item [$(c'_2)$] $\dfrac{\partial y}{\partial t} + Ay + f'(\bar y)y = u, y(0) = 0;$ 

\item[$(c'_3)$] $g_y[x, t]y(x, t) + g_u[x, t]u(x, t) \in T((-\infty, 0], g[x,t])$ for a.a. $(x, t) \in Q$.
\end{itemize}

The following theorem gives second-order sufficient conditions for locally optimal solutions to problem \eqref{P1}--\eqref{P4}. 

\begin{theorem} \label{Theorem 2} Suppose that $(H1), (H2), (H3)$ and $(H4)$ are valid, there exist the multipliers $\varphi \in W^{1, 2}(0, T; D, H) \cap L^\infty(Q)$ and $e \in L^\infty(Q)$ corresponding to $(\bar y, \bar u)\in\Phi$ satisfies conclusions $(i), (ii)$ and $(iii)$ of Theorem \ref{Theorem 1} and the following strictly second-order condition:
 \begin{align}\label{StrictSOSCond}
D^2_{(y,u)}\mathcal{L}(\bar y, \bar u, \varphi, e)[(y, v), (y, v)]>0\quad \forall (y, v)\in\mathcal{C}_2[(\bar y, \bar u)]\setminus\{(0,0)\}.
\end{align}  
Furthermore, there exists a number $\Lambda > 0$ such that 
\begin{align}\label{CoerciveCond}
    L_{uu}[x, t] + e(x, t)g_{uu}[x, t] \geq \Lambda \quad {\rm a.e.} \ (x, t) \in Q.
\end{align}
Then there exist number $\epsilon>0$ and $\kappa > 0$ such that 
\begin{align}
    J(y, u)\geq J(\bar y, \bar u) +\kappa \|u-\bar u\|_{L^2(Q)}^2\quad \forall (y, u)\in\Phi\cap[B_Y(\bar y, \epsilon)\times B_U(\bar u, \epsilon)], 
\end{align} where $B_Y(\bar y, \epsilon)$ and $B_U(\bar u, \epsilon)$ are  balls in $Y$ and $L^\infty(Q)$, respectively.
\end{theorem}
\begin{proof} Suppose to the contrary that the conclusion were
false. Then, we could find sequences $\{(y_n, u_n)\}\subset \Phi$ and $\{\gamma_n\}\subset \mathbb{R}_+$ such that $y_n\to \bar y$ in $Y$, $u_n\to \bar u$ in $L^\infty(Q)$,  $\gamma_n\to 0^+$ and
\begin{equation}\label{Iq1}
J(y_n, u_n)< J(\bar y, \bar u) +\gamma_n\|u_n-\bar u\|_{L^2(Q)}^2. 
\end{equation} If $u_n=\bar u$ then by the uniqueness we have $y_n=\bar y$. This leads to  $J(\bar y, \bar u)< J(\bar y, \bar u) + 0$ which is absurd. Therefore, we can assume that $u_n\ne \bar u$ for  all $n\geq 1$.

Define $t_n=\|u_n- \bar u\|_{L^2(Q)}$, $\zeta_n= \frac{y_n-\bar y}{t_n}$ and $v_n= \frac{u_n-\bar u}{t_n}$. Then $t_n\to 0^+$ and $\|v_n \|_{L^2(Q)}=1$. By the reflexivity of  $L^2(Q)$, we may assume that $v_n \rightharpoonup v$ in $L^2(Q)$. From the above, we have
\begin{equation}\label{KeyIq2}
J(y_n, u_n)- J(\bar y, \bar u)\leq t_n^2 \gamma_n \leq o(t_n^2).
\end{equation} We want to show that $z_n$ converges weakly to some
$z$ in $W^{1,2}(0, T; D, H)$. Since $(y_n,
u_n)\in\Phi$ and $(\bar y, \bar u)\in\Phi$, we have
\begin{align*}
&\frac{\partial y_n}{\partial t}+ Ay_n +f(y_n) = u_n,\ y_n(0)=y_0\\
&\frac{\partial \bar y}{\partial t} +A\bar y + f(\bar y)=\bar u,\ \bar y_n(0)= y_0.
\end{align*} This implies that
\begin{align}\label{Zn}
     \frac{\partial (y_n-\bar y)}{\partial t}+ A(y_n-\bar y) +f(y_n)-f(\bar y) = u_n-\bar u,\ (y_n-\bar y)(0)=0
\end{align}
 By Taylor's expansion, we have 
$$
f(y_n)-f(\bar y)=f'(\bar y +\theta_n(y_n-\bar y)) (y_n-\bar y),\ 0\leq\theta_n(x, t)\leq 1. 
$$ Such a function $\theta_n$ does exist and it is measurable. Indeed, we consider the set-valued map $H_n: Q\to 2^{\mathbb{R}}$ by setting
$$
H_n(x, t)=\{\theta\in [0, 1]: f(y_n(x,t))-f(\bar y(x,t))=f'(\bar y(x,t) +\theta(y_n(x,t)-\bar y(x,t))(y_n(x,t)-\bar y(x,t))\}.
$$ By \cite[Theorem 8.2.9]{Aubin}, $H_n$ is measurable and has a measurable selection, that is, there exists a measurable mapping $\theta_n: Q\to [0,1]$ such that $\theta_n(x, t)\in H_n(x, t)$ for a.a. $(x, t)\in Q.$

Note that $y_n\to \bar y$ in $L^\infty(Q)$ so there exists $M>0$ such that $\|\bar y +\theta_n(y_n-\bar y)\|_{L^\infty(Q)}\leq M$. By $(H2)$, there exists $k_M>0$ such that 
$$
|f'(\bar y +\theta_n(y_n-\bar y))|\leq k_M|\bar y +\theta_n(y_n-\bar y))| +|f'(0)|. 
$$ Hence $f'(\bar y +\theta_n(y_n-\bar y))\in L^\infty(Q)$. From \eqref{Zn}, we get
\begin{align}\label{Zn2}
     \frac{\partial \zeta_n}{\partial t}+ A\zeta_n +f'(\bar y+\theta_n (y_n-\bar y)\zeta_n = v_n,\ \zeta_n(0)=0.
\end{align} By \cite[Theorem 5, p.360]{Evan}, there exists a constant $C>0$ such that 
$$
\|\zeta_n\|_{W^{1,2}(0,T; D, H)}\leq  C\|v_n\|_{L^2(Q)}=C.
$$ Hence $\{z_n\}$ is bounded in $W^{1,2}(0,T; D, H)$. Without loss of generality, we may assume that $\zeta_n\rightharpoonup \zeta$ in $W^{1,2}(0,T, D, H)$. By Aubin's lemma, the embedding $W^{1,2}(0, T; D, H)\hookrightarrow L^2([0, T], V)$ is compact. Therefore, we may assume that $\zeta_n\to \zeta$ in $L^2([0, T], V)$. Particularly, $\zeta_n\to \zeta$ in $L^2(Q)$. We now  use the procedure in the proof of \cite[Theorem 3, p. 356]{Evan}. By  passing to the limit,   we obtain from  \eqref{Zn2} that
\begin{align}\label{Z-Eq-C2}
     \frac{\partial z}{\partial t}+ A\zeta +f'(\bar y)\zeta = v,\ \zeta(0)=0.
\end{align}
 Next we claim that $(\zeta, v)\in\mathcal{C}_2[(\bar y, \bar u)]$. Indeed, by a Taylor expansion, we have from \eqref{KeyIq2} that
\begin{equation}\label{CriticalCon1}
\int_Q L(x,t, \bar y +\theta(y_n-\bar y), u_n)\zeta_n dxdy +\int_Q L_u(x, t, \bar y,  \bar u +\xi(u_n-\bar u)) v_n dxdt\leq \frac{o(t^2_n)}{t_n},
\end{equation}  where $0\leq \theta, \xi\leq 1$.  Since $y_n\to \bar y$  and $u_n\to \bar u$ in $L^\infty (Q)$, there exists $M>0$ such that 
$$
|u_n(x, t))|, |\bar y(x, t) +\theta(x,t)(y_n(x,t)-\bar y(x,t)| \leq M+ \|\bar y\|_{L^\infty(Q)} +\|\bar u\|_{L^\infty(Q)}. 
$$ By $(H3)$, there exists $k_M>0$ such that 
\begin{align*}
&|L_y(x, t, \bar y(x, t) +\theta(x,t)(y_n(x,t)-\bar y(x,t), u_n(x, t))-L_y(x, t, \bar y(x, t), \bar u(x, t))|\\
&\leq k_M(|y_n(x, t)-\bar y(x, t) |+ |u_n(x, t)-\bar u(x, t)|).
\end{align*} This implies that 
$$
\|L_y(\cdot, \cdot, \theta(y_n-\bar y), u_n)-L_y(\cdot, \cdot, \bar y, \bar u)\|_{L^\infty(Q)}\to 0\quad {\rm as}\quad n\to\infty
$$ Similarly, we have 
$$
\|L_y(\cdot, \cdot, \bar y, \xi(u_n-\bar u)-L_y(\cdot, \cdot, \bar y, \bar u)\|_{L^\infty(Q)}\to 0\quad {\rm as}\quad n\to\infty
$$  Letting $n\to\infty$  we deduce from \eqref{CriticalCon1} that
\begin{align}\label{c1}
\int_Q(L_y[x,t]\zeta(x,t)+ L_u[x,t]v(x,t))dxdt\leq 0. 
\end{align} Hence conditions $(c_1')$  and $(c_2')$ of  $\mathcal{C}_2[(\bar y, \bar u)]$ are valid.  It remains to verify  $(c_3')$. Let us define 
$$
G(y, u)(x, t) =g(x, t, y(x, t), u(x, t)). 
$$ Since $(y_n, u_n)\in\Phi$, we have $G(y_n, u_n) \in K_\infty$ for $n = 1, 2, ...$. Hence 
$$
G(y_n, u_n) - G(\bar y, \bar u) \in K_\infty - G(\bar y, \bar u)\subset K_2-G(\bar y, \bar u),
$$ where 
$$
K_2:=\{v\in L^2(Q): v(x, t)\leq 0\quad {\rm a.a.}\ (x,t)\in Q\}. 
$$
By the Taylor expansion and the definitions of $\zeta_n, v_n$, we have
\begin{align*}
    G_y(\bar y +\theta_1(y_n-\bar y), u_n)\zeta_n + G_u(\bar y, \bar u+\xi_1(u_n-\bar u))v_n & \in \frac{1}{t_n} \bigg( K_2 - G(\bar y, \bar u) \bigg) \\
    &\subseteq {\rm cone}(K_2 - G(\bar y, \bar u))  \\
    &\subseteq T_{L^2(Q)}(K_2; G(\bar y, \bar u))
\end{align*}
where $0\leq \theta_1, \xi_1\leq 1$, $T_{L^2(Q)}(K_2; G(\bar y, \bar u))$ is the tangent cone to $K_2$ at $G(\bar y, \bar u)$ in $L^2(Q)$. Since $T_{L^2}(K_2; G(\bar y, \bar u))$ is a closed convex subset of $L^2(Q)$, it is a weakly closed set of $L^2(Q)$ and so $T_{L^2}(K_2; G(\bar y, \bar u))$ is sequentially weakly closed. By a simple argument, we can show that $G_y(\bar y +\theta_1(y_n-\bar y), u_n)\zeta_n + G_u(\bar y, \bar u+\xi_1(u_n-\bar u))v_n$ converges weakly to    $G_y(\bar y, \bar u)\zeta + G_u(\bar y, \bar u)v$ in $L^2(Q)$. Letting $n\to\infty$, we obtain from the above that 
\begin{align}
    G_y(\bar y, \bar u)\zeta + G_u(\bar y, \bar u)v \in T_{L^2}(K_2; G(\bar y, \bar u)).
\end{align} By Lemma 2.4 in \cite{Kien-2017}, we have 
$$
T_{L^2}(K_2; G(\bar y, \bar u))=\{ v\in L^2(Q): v(x, t)\in T((-\infty, 0]; g[x,t])\quad {\rm a.a.}\ (x, t)\in Q\}. 
$$ It follows that 
$$
g_y[x,t]\zeta(x,t) + g_u[x,t]v(x,t)\in T((-\infty, 0]; g[x,t])\quad {\rm a.a.}\ (x, t)\in Q. 
$$ Hence $(c_3')$ is satisfied. Consequently, $(\zeta,v)\in\mathcal{C}_2[(\bar y, \bar u)]$ and the claim is justified. 

Let us claim  that $(\zeta, v)=0$.  In fact, from the assertion $(iii)$ of Theorem {\ref{Theorem 1}}, we have $\langle e, G(y_n, u_n) - G(\bar y, \bar u)\rangle \leq 0$.  By \eqref{LagrangeFunction2} and  \eqref{Zn}, we have 
\begin{align}
&\mathcal{L}(y_n, u_n, \varphi, e)-\mathcal{L}(\bar y,\bar u, \varphi, e)=J(y_n, u_n)-J(\bar y, \bar u)  \notag\\ 
&+\int_Q [(y_n-\bar y)[-\frac{\partial \varphi}{\partial t} +A^*\varphi] +\varphi (f(y_n)-f(\bar y)) -\varphi(u_n-\bar u)  + e(G(y_n, u_n) - G(\bar y, \bar u))]dxdt\label{LagrangeFunction3}\\
&=J(y_n, u_n)-J(\bar y, \bar u)\notag \\  
&+\int_Q[\varphi(t)\big(\frac{\partial (y_n-y)}{\partial t} +A(y_n-\bar y) +(f(y_n)-f(\bar y)- (u_n-\bar u)\big)+ e(G(y_n, u_n) - G(\bar y, \bar u))] dxdt\notag\\
&=J(y_n, u_n)-J(\bar y, \bar u) + \int_Q  e(G(y_n, u_n) - G(\bar y, \bar u)) dxdt\notag \\
&\leq J(y_n, u_n)-J(\bar y, \bar u)\leq o(t_n^2).\notag
\end{align} Using   a second-order Taylor expansion and  the equality  $D_{(y,u)}\mathcal{L}(\bar y, \bar u, \varphi, e)=0$, we have from  \eqref{LagrangeFunction3} that
\begin{align*}
&o(t_n^2)\geq \mathcal{L}(y_n, u_n, \varphi, e)-\mathcal{L}(\bar y,\bar u, \varphi, e)=\\
&\int_Q \big(t_nL_y[x,t]\zeta_n + t_n^2L_{yu}(x,t, \bar y, \bar u+ \theta_2(u_n-\bar u))\zeta_n v_n +\frac{t_n^2}2L_{yy}(x,t,\bar y+\theta_3(y_n-\bar y), u_n)\zeta_n^2\big)dxdt\\ 
&+\int_Q\big(t_nL_u[x, t] v_n+ \frac{t_n^2}2L_{uu}(x, t, \bar y, \bar u +\xi_2(u_n-\bar u)) v_n^2\big)dxdt \\
&+\int_Q t_n(-\varphi_t +A^*\varphi +f'(\bar y)\varphi) \zeta_n +\frac{t_n^2}{2} \varphi f(\bar y +\theta_4(y_n-\bar y)) \zeta_n^2 + t_n \varphi v_n \\
&+ \int_Q e\big(t_ng_y[x,t]\zeta_n + t_n^2g_{yu}(x,t, \bar y, \bar u+ \theta_5(u_n-\bar u))\zeta_n v_n +\frac{t_n^2}2g_{yy}(x,t,\bar y+\theta_6(y_n-\bar y), u_n)\zeta_n^2\big)dxdt\\ 
&+\int_Qe\big(t_ng_u[x, t] v_n+ \frac{t_n^2}2g_{uu}(x, t, \bar y, \bar u +\xi_3(u_n-\bar u)) v_n^2\big)dxdt \\
&=\frac{t_n^2}2\int_Q \big(2L_{yu}(x,t, \bar y, \bar u+ \theta_2(u_n-\bar u))\zeta_n v_n +L_{yy}(x,t,\bar y+\theta_3(y_n-\bar y), u_n)\zeta_n^2\big)dxdt\\
&+\frac{t_n^2}2\int_Q e\big(2g_{yu}(x,t, \bar y, \bar u+ \theta_5(u_n-\bar u))\zeta_n v_n+ g_{yy}(x,t,\bar y+\theta_6(y_n-\bar y), u_n)\zeta_n^2\big)dxdt\\ 
&+\frac{t_n^2}2\int_Q \big(L_{uu}(x, t, \bar y, \bar u +\xi_2(u_n-\bar u)) + eg_{uu}(x, t, \bar y, \bar u +\xi_3(u_n-\bar u)) v_n^2\big) v_n^2dxdt,
\end{align*} where $0\leq \theta_i, \xi_j\leq 1$. It follows that 
\begin{align}
\frac{o(t_n^2)}{t_n^2}&\geq \int_Q \big(2L_{yu}(x,t, \bar y, \bar u+ \theta_2(u_n-\bar u))\zeta_n v_n +L_{yy}(x,t,\bar y+\theta_3(y_n-\bar y), u_n)\zeta_n^2\big)dxdt\notag\\
&+\int_Q e\big(2g_{yu}(x,t, \bar y, \bar u+ \theta_5(u_n-\bar u))\zeta_n v_n+ g_{yy}(x,t,\bar y+\theta_6(y_n-\bar y), u_n)\zeta_n^2\big)dxdt\notag\\ 
&+\int_Q \big(L_{uu}(x, t, \bar y, \bar u +\xi_2(u_n-\bar u)) + eg_{uu}(x, t, \bar y, \bar u +\xi_3(u_n-\bar u))\big) v_n^2dxdt.\label{NonnegativeCond4}
\end{align} Since $\|u_n- \bar u\|_{L^\infty(Q)}\to 0$, $\|y_n- \bar y\|_{L^\infty(Q)}\to 0$ and $(H3)$, we get   
$$
\|L_{uu}(\cdot,\cdot, \bar y, \bar u +\xi_2(u_n-\bar u)) + eg_{uu}(\cdot, \cdot, \bar y, \bar u +\xi_3(u_n-\bar u)- (L_{uu}[\cdot, \cdot] + e g_{uu}[\cdot, \cdot])\|_{L^\infty(Q)}\to 0.
$$ From this and  \eqref{CoerciveCond}, we see that there exists $n_0>0$ such that 
\begin{align}\label{CoerciveCond2}
 L_{uu}(x,t, \bar y, \bar u +\xi_2(u_n-\bar u)) + eg_{uu}(x, t, \bar y, \bar u +\xi_3(u_n-\bar u)> \frac{\Lambda}2   
\end{align} for all $n>n_0$ and a.a. $(x, t)\in Q$. Also, by \eqref{CoerciveCond}, the  functional 
\begin{align}
w \mapsto  \int_Q \big(L_{uu}[x,t] + eg_{uu}[x, t])\big) w^2dxdt
\end{align} is convex and so it is sequentially lower semicontinous. Therefore, we have 
\begin{align*}
&\lim_{n\to\infty}\int_Q\big(L_{uu}(x,t, \bar y, \bar u +\xi_2(u_n-\bar u)) + eg_{uu}(x, t, \bar y, \bar u +\xi_3(u_n-\bar u))\big)v_n^2dxdt\\
&=\lim_{n\to\infty}\int_Q\big(L_{uu}(x,t, \bar y, \bar u +\xi_2(u_n-\bar u)) + eg_{uu}(x, t, \bar y, \bar u +\xi_3(u_n-\bar u)- L_{uu}[x,t] -e g_{uu}[x,t])\big)v_n^2 dxdt\\
&+\lim_{n\to\infty}\int_Q\big( L_{uu}[x,t] + eg_{uu}[x,t]\big) v_n^2 dxdt\geq \int_Q (L_{uu}[x,t] + eg_{uu}[x, t]) v^2 dxdt. 
\end{align*} From this and letting $n\to\infty$, we obtain from \eqref{NonnegativeCond4} that 
$$
0\geq D^2_{(y,u)}\mathcal{L}(\bar y,\bar u, \varphi, e)[(\zeta, v), (\zeta, v)].
$$ Combining this with \eqref{StrictSOSCond}, we conclude that $(\zeta, v)=(0, 0)$.

Finally, from \eqref{NonnegativeCond4}, $\|v_n\|_{L^2(Q)}=1$ and \eqref{CoerciveCond2}, we have 
\begin{align*}
\frac{o(t_n^2)}{t_n^2}&\geq \int_Q \big(2L_{yu}(x,t, \bar y, \bar u+ \theta_2(u_n-\bar u))\zeta_n v_n +L_{yy}(x,t,\bar y+\theta_3(y_n-\bar y), u_n)\zeta_n^2\big)dxdt\\
&+\int_Q e\big(2g_{yu}(x,t, \bar y, \bar u+ \theta_5(u_n-\bar u))\zeta_n v_n+ g_{yy}(x,t,\bar y+\theta_6(y_n-\bar y), u_n)\zeta_n^2\big)dxdt +\frac{\Lambda}2. 
\end{align*} 
By letting $n\to\infty$ and using $(\zeta, v)=(0, 0)$, we get $0\geq \frac{\Lambda}2$ which is impossible. Therefore, the proof of  theorem is complete.   
\end{proof}

\section{H\"{o}lder continuity of multipliers and optimal solutions}

In this section we show that under stronger hypotheses on $y_0$, $L, g$ and $\partial\Omega$, then the multipliers and optimal solution are H\"{o}lder continuous. Namely, we require that $\phi$ is locally Lipschitz in all variable.  

\noindent $(H3')$ $\phi$ is a continuous  function and  for each $(x,t)\in \Omega\times [0, T]$, $\phi(x, t, \cdot, \cdot)$ is of class $C^2$ and satisfies the following property:  for each $M>0$, there exists $k_{\psi,M}>0$ such that 
\begin{align*}
    &|\phi(x_1, t_1, y_1, u_1)-\phi(x_2, t_2, y_2, u_2)|+|\phi_y(x_1, t_1, y_1, u_1)-\phi_y(x_2, t_2, y_2, u_2)|\\
    &+|\phi_u(x_1, t_1, y_1, u_1)-\phi_u(x_2, t_2, y_2, u_2)| \leq k_{\phi, M} (|x_1-x_2|+ |t_1-t_2|+ |y_1-y_2|+ |u_1-u_2|),\\
    &{\rm and} \\
    &|\phi_{yy}(x, t, y_1, u_1)-\phi_{yy}(x, t, y_2, u_2)|+|\phi_{yu}(x, t, y_1, u_1)-\phi_{yu}(x, t, y_2, u_2)|\\
    &+ |\phi_{uu}(x, t, y_1, u_1)-\phi_{uu}(x, t, y_2, u_2)|\leq k_{\phi,M}(|y_1-y_2|+ |u_1-u_2|)
\end{align*}  for all $(x, t), (x_i, t_i)\in Q$ and $(y_i, u_i)\in \mathbb{R}\times\mathbb{R}$ satisfying $|y_i|, |u_i|\leq M$ with $i=1,2$.

\noindent $(H4')$  There exist nunmbers $\gamma_1>0$ and $\gamma_2>0$ such that 
\begin{align}
& L_{uu}(x, t, y, u)\geq \gamma_1\quad \forall (x,t, y, u)\in Q\times \mathbb{R}\times\mathbb{R},\\
 & g_u(x, t, y, u)\geq \gamma_2\quad \forall (x,t, y, u)\in Q\times \mathbb{R}\times\mathbb{R}.
\end{align}

Recall that the boundary $\partial\Omega$ satisfies the property of positive geometric density if there exist number $\alpha^*\in (0, 1)$ and $R_0>0$ such that for all $x_0\in\partial \Omega$ and all $R\leq R_0$, one has
$$
|\Omega\cap B(x_0, R)|\leq (1-\alpha^*) |B(x_0, R)|,
$$ where $|C|$ denote Lebessgue measure of a measurable set $C$ in $\mathbb{R}^N$. It is known that if $\partial \Omega$ is of class $C^{1,1}$ then it satisfies this property. Moreover,  by using the theorem of supporting hyperplane \cite[Theorem 1.5, p.19]{H.Tuy}, we can show that if $\Omega \subset \mathbb R^N$ is a nonempty convex set, then  $\partial\Omega$ satisfies the property of positive geometric density with $\alpha^*=1/2$ and $R_0>0$.

The property of positive geometric density was introduced by E.Di.Benedetto in \cite{E.Di.Benedetto}. In the next theorem we are going to use this property for the H\"{o}lder continuity of solutions to the semilinear parabolic equations. 

\begin{theorem} \label{Theorem-Regularity} Suppose $y_0\in H^2(\Omega)\cap H_0^1(\Omega)$, hypotheses $(H1), H(2), (H3'), (H4')$, and $\partial \Omega$ is of class $C^{1,1}$ or $\Omega$ is convex.  Assume that there exist  multipliers  $e \in L^\infty(Q)$ and $\varphi \in W^{1, 2}(0, T; D, H) \cap L^\infty(Q)$ which satisfy conditions $(i)-(iv)$ of Theorem \ref{Theorem 1}. Then $\varphi, e$, $\bar y$ and $\bar u$ are H\"{o}lder continuous on $\bar Q$.
\end{theorem}
\begin{proof} By the Sobolev inequality (see \cite[Theorem 6, p. 270]{Evan}) we have
$y_0\in C^{0, \alpha}(\Omega)\cap H_0^1(\Omega)$ for some $\alpha\in (0, 1]$.  Since implications $(H3')\Rightarrow (H3)$ and $(H4')\Rightarrow (H4)$ are valid, the existence of multipliers $\varphi$ and $e$ follows from Theorem \ref{Theorem 1}.   It remains to show that $\bar y, \bar u, \varphi$ and $e$ are H\"{o}lder continuous.  

\noindent \textbf{Step 1.}  Showing that $\bar y$ and $\varphi$ are H\"{o}lder continuous functions.

Since $\bar y \in L^\infty(Q)$, $\bar y$ is as a bounded solution of the parabolic (1.2)-(1.3). From this, $y_0\in C^{0, \alpha}(\bar \Omega)$ and  \cite[Corollary 0.1]{E.Di.Benedetto}, we see that  $\bar y$ is $\alpha-$H\"{o}lder continuous on $\bar Q$ with exponent $\alpha \in (0,1)$. By setting $\widehat \varphi(t) = \varphi(T - t)$, then $\widehat \varphi$ is unique solution of equation
    \begin{align}
        \frac{\partial \widehat \varphi (t)}{\partial t} + A^* \widehat \varphi(t) + f'(\bar y (x, T - t)) \widehat \varphi (t) = - L_y[x, T-t] - e(x, T-t)g_y[x, T-t], \quad  \widehat \varphi (., 0) = 0.
    \end{align}
    By the same argument, we have from \cite[Corollary 0.1]{E.Di.Benedetto} that $\widehat \varphi$ is also $\alpha-$H\"{o}lder continuous on $\bar Q$. It follows that   $\varphi$ is $\alpha-$H\"{o}lder continuous. 

    \noindent \textbf{Step 2.} Showing that there is a function $\phi : \bar Q \times \mathbb R \to \mathbb R$ such that for every $y \in \mathbb R$, $\phi (., ., y)$ is Lipschitz, and for all $(x, t) \in \bar Q$ the function $\phi(x, t, .)$ is locally Lipschitz, i.e., 
    \begin{align*}
    &|\phi(x_1, t_1, y)-\phi(x_2, t_2, y)|\leq L_{\phi} (|x_1-x_2| +|t_1-t_2|),\quad \forall (x_i, t_i, y)\in Q\times\mathbb{R}, i=1,2\\
    &| \phi(x, t, y_1) - \phi(x, t, y_2)| \leq L_{\phi, M} |y_1 - y_2| \quad \text{for all} \ \ |y_1|, |y_2| \le M,
    \end{align*}
    where $L_\phi, L_{\phi, M} > 0$ are  constants, and moreover, it holds
        \begin{align}
        \label{Holder-2.1}
            g(x, t, y, u) \begin{cases}
                = 0  \Leftrightarrow  u = \phi(x, t, y),\\
                < 0  \Leftrightarrow  u < \phi(x, t, y),\\
                > 0  \Leftrightarrow  u > \phi(x, t, y).
            \end{cases}
        \end{align}
        This fact is established in the same way of the proof of \cite[Lemma 5.2]{Rosch2}.
        
    \noindent  \textbf{Step 3.}  Showing that $(x, t) \mapsto g_u[x, t]e(x, t)$ belongs to $C^{0, \alpha}(\bar Q)$.\\
    To do this, it's sufficient to  show that
    \begin{align}
    \label{Holder-5.5}
        g_u[x, t]e(x, t) = {\rm max} \bigg\{ 0, \ \varphi(x, t)  -  L_u\bigg(x,   t,   \bar y(x, t),  \phi(x, t, \bar y(x,t))\bigg) \bigg\} \quad \forall \ (x, t) \in Q
    \end{align}
    and  the function
    \begin{align}
    \label{Holder-5.6}
        (x, t)  \mapsto   L_u\big(x,   t,   \bar y(x, t),  \phi(x, t, \bar y(x,t))\big)
    \end{align} is H\"{o}lder continuous.  For this we  consider two cases:
 
 \noindent \textit{Case 1.} $(x, t) \in Q^+:= \{(x, t) \in Q: e(x, t) > 0\}$. 
 
 Then we have $e(x, t) > 0$ for all $(x, t) \in Q^+$. It follows from this and the complementary slackness condition $(iii)$ of Theorem \ref{Theorem 1}, 
    $g[x, t] = 0$. From this and \eqref{Holder-2.1}, we get 
    \begin{align}
    \label{Holder5.2}
        u(x, t) = \phi(x, t, y(x, t)) \quad {\rm a.a. \ on} \ Q^+. 
    \end{align}
    From $(ii)$ of Theorem \ref{Theorem 1}, 
    \begin{align}
    \label{Holder5.3}
        \varphi(x, t)  -  L_u\big(x,   t,   \bar y(x, t),  u(x, t)\big) = g_u[x, t]e(x, t) \quad {\rm a.a. \ on} \ Q^+. 
    \end{align}
    Inserting (\ref{Holder5.2}) onto (\ref{Holder5.3}),
    \begin{align}
        \varphi(x, t)  -  L_u\big(x,   t,   \bar y(x, t),  \phi(x, t, \bar y(x,t))\big) = g_u[x, t]e(x, t) \ge \gamma_2 e(x, t)   > 0 \quad {\rm a.a. \ on} \ Q^+. 
    \end{align}
    Hence we obtain
    \begin{align*}
        g_u[x, t]e(x, t) = {\rm max} \bigg\{ 0, \ \varphi(x, t)  -  L_u\bigg(x,   t,   \bar y(x, t),  \phi(x, t, \bar y(x,t))\bigg) \bigg\} \quad \forall \ (x, t) \in Q^+.
    \end{align*}

    \noindent \textit{Case 2.}  $(x, t) \in Q \backslash Q^+ = \{(x, t) \in Q: e(x, t) = 0\}$. 
    
    Then, from $(ii)$ of Theorem \ref{Theorem 1}, 
    \begin{align}
    \label{Holder5.4}
        \varphi(x, t) - L_u\big(x,   t,   \bar y(x, t),  u(x, t)\big) = g_u[x, t]e(x, t) = 0 \quad {\rm a.a. \ on} \ Q \backslash Q^+. 
    \end{align}
    Moreover, since $g[x, t] \le 0$ a.a. on $Q$, and from (\ref{Holder-2.1}), we have 
    \begin{align}
        u(x, t) \le \phi(x, t, y(x, t)) \quad {\rm a.a. \ on} \ Q \backslash Q^+.
    \end{align}
    Combining this fact and the monotonicity property of function $L_u(x, t, y, .)$ from assumption $(H4')$ we have 
    \begin{align}
        L_u\big(x,   t,   \bar y(x, t),  u(x, t) \le L_u\big(x,   t,   \bar y(x, t),  \phi(x, t, \bar y(x,t))\big)  \quad {\rm a.a. \ on} \ Q \backslash Q^+.
    \end{align}
    Together  with (\ref{Holder5.4}), this implies 
    \begin{align}
         \varphi(x, t)  -  L_u\big(x,   t,   \bar y(x, t),  \phi(x, t, \bar y(x,t))\big) \le 0 \quad {\rm a.a. \ on} \ Q \backslash Q^+. 
    \end{align}
    Hence we get
    \begin{align*}
        g_u[x, t]e(x, t) = {\rm max} \bigg\{ 0, \ \varphi(x, t)  -  L_u\bigg(x,   t,   \bar y(x, t),  \phi(x, t, \bar y(x,t))\bigg) \bigg\} \quad \forall \ (x, t) \in Q \backslash Q^+.
    \end{align*} Thus (\ref{Holder-5.5}) is established. 
    
    We now check (\ref{Holder-5.6}). Indeed, for any $(x_1, t_1), (x_2, t_2) \in Q$, we have 
    \begin{align*}
      I :=   &|L_u\big(x_1,   t_1,   \bar y(x_1, t_1),  \phi(x_1, t_1, \bar y(x_1,t_1))\big) - L_u\big(x_2,   t_2,   \bar y(x_2, t_2),  \phi(x_2, t_2, \bar y(x_2,t_2))\big)| \\
        &= |L_u\big(x_1,   t_1,   \bar y(x_1, t_1),  \phi(x_1, t_1, \bar y(x_1,t_1))\big) - L_u\big(x_2,   t_2,   \bar y(x_1, t_1),  \phi(x_1, t_1, \bar y(x_1,t_1))\big)|  \\
        &+|L_u\big(x_2,   t_2,   \bar y(x_1, t_1),  \phi(x_1, t_1, \bar y(x_1,t_1))\big) - L_u\big(x_2,   t_2,   \bar y(x_2, t_2),  \phi(x_2, t_2, \bar y(x_2,t_2))\big)| \\
        &=: I_1 + I_2.
    \end{align*}
    Since the function $L_u(., ., y, u)$ is Lipschitz for all fixed $y, u \in \mathbb R$, there exists a constant $k_{L, M}$ such that 
    \begin{align}
        I_1 \le k_{L, M}|(x_2 - x_1, t_2 - t_1)|.
    \end{align}
    Using assumption $(H3')$ with $M = {\rm max}(\|\bar y\|_{L^\infty(Q)}, \|\bar u\|_{L^\infty(Q)}) + 1$, we have 
    \begin{align}
        I_2 &\le k_{L, M} \bigg(|\bar y(x_1, t_1) - \bar y(x_2, t_2)| + |\phi(x_1, t_1, \bar y(x_1,t_1)) - \phi(x_2, t_2, \bar y(x_2,t_2))| \bigg) 
    \end{align}
    Since $\bar y \in C^{0, \alpha}(\bar Q)$, $\phi(., ., \bar y) \in C^{0, 1}(\bar Q)$ and $\phi(x, t, .)$ is locally Lipschitz,
      \begin{align}
        I_2 &\le k_{L, M} \bigg( H_{\bar y}|(x_2 - x_1, t_2 - t_1)|^\alpha         + |\phi(x_1, t_1, \bar y(x_1,t_1)) - \phi(x_2, t_2, \bar y(x_1,t_1))| \nonumber \\
        & \hspace{5.5cm} + |\phi(x_2, t_2, \bar y(x_1,t_1)) - \phi(x_2, t_2, \bar y(x_2,t_2))| \bigg) \nonumber \\
        &\le k_{L, M} \bigg( H_{\bar y}|(x_2 - x_1, t_2 - t_1)|^\alpha + L_{1, \phi}|(x_2 - x_1, t_2 - t_1)| + L_{2, \phi}|\bar y(x_2,t_2) - \bar y(x_1,t_1)|  \bigg) \nonumber \\
        &\le k_{L, M} \bigg[ (L_{2, \phi} + 1)H_{\bar y}|(x_2 - x_1, t_2 - t_1)|^\alpha + L_{1, \phi}|(x_2 - x_1, t_2 - t_1)|  \bigg] \nonumber \\
        &= k_{L, M} \bigg[ (L_{2, \phi} + 1)H_{\bar y} + L_{1, \phi}|(x_2 - x_1, t_2 - t_1)|^{1 -  \alpha}  \bigg]|(x_2 - x_1, t_2 - t_1)|^\alpha \nonumber \\
        &\le k_{L, M} \bigg[ (L_{2, \phi} + 1)H_{\bar y} + L_{1, \phi}({\rm diam}Q)^{1 -  \alpha}  \bigg]|(x_2 - x_1, t_2 - t_1)|^\alpha.
    \end{align}
    Here ${\rm diam}(Q):= {\rm sup}_{z_1, z_2 \in Q, z_1 \ne z_2}|z_1 - z_2|$, $Q$ is bounded, ${\rm diam}(Q) < +\infty$. 
    Hence, we obtain 
    \begin{align}
        I \le H_{L_u}|(x_2 - x_1, t_2 - t_1)|^\alpha 
    \end{align}
    for some constant $H_{L_u}>0$ which is independent of $x_i$ and $t_i$. It follows that the function $ (x, t)  \mapsto   L_u\big(x,   t,   \bar y(x, t),  \phi(x, t, \bar y(x,t))\big)$
       belongs to $ C^{0, \alpha}(\bar Q)$. Thus, $g_u[., .]e(., .) \in C^{0, \alpha}(\bar Q)$.

\noindent \textbf{Step 4.}  Showing that $\bar u \in C^{0, \alpha}(\bar Q)$.\\
    Fixing any $(x_i, t_i)\in Q$ with $i=1,2$ and using a Taylor' expansion, we have
      \begin{align*}
      &|L_u(x_1, t_1, \bar y(x_1, t_1), \bar u(x_1, t_1))-L_u(x_1, t_1, \bar y(x_1, t_1), \bar u(x_2, t_2)|\\
      &=|L_{uu}(x_1, t_1, \bar y(x_1, t_1), \bar u(x_2, t_2) +\theta(\bar u(x_1, t_1)-\bar u(x_2, t_2))||(\bar u(x_1, t_1)-\bar u(x_2, t_2))|\\
      &\geq \gamma_1 |(\bar u(x_1, t_1)-\bar u(x_2, t_2))|,
      \end{align*}
where $\theta\in [0,1]$ and the last inequality follows from $(H4')$. Combining this with equality $L_u[\cdot, \cdot] = \varphi - g_u[\cdot, \cdot]e$, we get 
\begin{align*}
\gamma_1 |(\bar u(x_1, t_1)-\bar u(x_2, t_2))|&\leq |L_u(x_1, t_1, \bar y(x_1, t_1), \bar u(x_1, t_1))-L_u(x_1, t_1, \bar y(x_1, t_1), \bar u(x_2, t_2))|\\
&\leq |L_u(x_1, t_1, \bar y(x_1, t_1), \bar u(x_1, t_1))-L_u(x_2, t_2, \bar y(x_2, t_2), \bar u(x_2, t_2)|\\
&+ |L_u(x_2, t_2, \bar y(x_2, t_2), \bar u(x_2, t_2)-L_u(x_1, t_1, \bar y(x_1, t_1), \bar u(x_2, t_2))|\\
&=|\varphi(x_1, t_1)-\varphi(x_2, t_2) -g_u[x_1, t_1]e(x_1, t_1) + g_u[x_2, t_2]e(x_2,t_2)|\\
&+|L_u(x_2, t_2, \bar y(x_2, t_2), \bar u(x_2, t_2)-L_u(x_1, t_1, \bar y(x_1, t_1), \bar u(x_2, t_2))|\\
&\leq |\varphi(x_1, t_1)-\varphi(x_2, t_2)| +|-g_u[x_1, t_1]e(x_1, t_1) + g_u[x_2, t_2]e(x_2,t_2)|\\
&+|L_u(x_2, t_2, \bar y(x_2, t_2), \bar u(x_2, t_2)) - L_u(x_1, t_1, \bar y(x_1, t_1), \bar u(x_2, t_2))|\\
&\leq (k_\varphi +k_{g, M})|(x_1, t_1)-(x_2, t_2)|^\alpha +k_{L, M}(|t_1-t_2| +|x_1-x_2|).
\end{align*} for some $k_\varphi>0, k_{g, M}>0, k_{L, M}>0$ which are independent of $(x_i, t_i)$. Hence $\bar u$ is H\"{o}lder continuous on $\bar Q$.

\noindent \textbf{Step 5.}  Showing that the Lagrange multiplier $e$ is H\"{o}lder continuous on $\bar Q$.\\
    Define $\widetilde g(., .) = g_u[., .]e$, then $\widetilde g(., .) \in C^{0, \alpha}(\bar Q)$. Taking any $(x_1, t_1), (x_2, t_2) \in \bar Q$, we have  
       \begin{align}
           | e(x_1, t_1) - e(x_2, t_2)| &= \bigg|\frac{\widetilde g(x_1, t_1)}{g_u[x_1, t_1]} - \frac{\widetilde g(x_2, t_2)}{g_u[x_2, t_2]}  \bigg| \nonumber \\
           &= \bigg | \frac{1}{g_u[x_1, t_1]}[\widetilde g(x_1, t_1)  -  \widetilde g(x_2, t_2)] + [\frac{1}{g_u[x_1, t_1]} - \frac{1}{g_u[x_2, t_2]}]\widetilde g(x_2, t_2) \bigg| \nonumber \\
           &\le \bigg | \frac{1}{g_u[x_1, t_1]}\bigg| |\widetilde g(x_1, t_1)  -  \widetilde g(x_2, t_2)| + \bigg| \frac{g_u[x_2, t_2] - g_u[x_1, t_1]}{g_u[x_1, t_1]g_u[x_2, t_2]} \bigg| |\widetilde g(x_2, t_2)|. \nonumber 
           \end{align}
           From this and the facts $g_u[., .] \geq \gamma_2 > 0$ a.a. on $\bar Q$, $\widetilde g \in C^{0, \alpha}(\bar Q) \subset C(\bar Q)$, we obtain 
           \begin{align}
           \label{Holder-Step5.1}
               | e(x_1, t_1) - e(x_2, t_2)| \le \frac{1}{\delta_0} \big|\widetilde g(x_1, t_1)  -  \widetilde g(x_2, t_2)\big|  +  \frac{\|\widetilde g\|_{C(\bar Q)}}{\delta_0^2} \big|g_u[x_2, t_2] - g_u[x_1, t_1] \big|.
           \end{align}
           Since $\widetilde g(., .) \in C^{0, \alpha}(\bar Q)$, there exists $H_{\widetilde g} > 0$,
           \begin{align}
           \label{Holder-Step5.2}
               \big|\widetilde g(x_1, t_1)  -  \widetilde g(x_2, t_2)\big| \le H_{\widetilde g}|(x_2 - x_1, t_2 - t_1)|^\alpha.
           \end{align}
           Also,
           \begin{align}
               \big|g_u[x_2, t_2] - g_u[x_1, t_1] \big|  
               &= \bigg|g_u(x_2, t_2, \bar y(x_2, t_2), \bar u(x_2, t_2)) - g_u(x_1, t_1, \bar y(x_1, t_1), \bar u(x_1, t_1))\bigg| \nonumber \\
               &= \bigg|g_u(x_2, t_2, \bar y(x_2, t_2), \bar u(x_2, t_2)) - g_u(x_1, t_1, \bar y(x_2, t_2), \bar u(x_2, t_2))\bigg| \nonumber \\ 
               &\quad + \bigg|g_u(x_1, t_1, \bar y(x_2, t_2), \bar u(x_2, t_2)) - g_u(x_1, t_1, \bar y(x_1, t_1), \bar u(x_1, t_1))\bigg|.\notag
           \end{align}
           From this, $(H3')$ and H\"{o}lder continuity of $\bar y$ and $\bar u$, we get
           \begin{align}
               &\big|g_u[x_2, t_2] - g_u[x_1, t_1] \big| \nonumber \\ 
               &\leq  k_{g, M}(|x_2-x_1|+|t_2-t_1|) + k_{g,M} \bigg( |\bar y(x_2, t_2) - \bar y(x_1, t_1) | + |\bar u(x_2, t_2) - \bar u(x_1, t_1) | \bigg)\notag\\
               &\leq k_{g, M}(|x_2-x_1|+|t_2-t_1|) + k_{g, M}(k_{\bar y} +k_{\bar u})|(x_2 - x_1, t_2 - t_1)|^\alpha \label{Holder-Step5.3} 
           \end{align}  
           Combining  (\ref{Holder-Step5.1}), (\ref{Holder-Step5.2}) and \eqref{Holder-Step5.3}, we obtain
           \begin{align}
               |e(x_2, t_2) - e(x_1, t_1)| \le H_e|(x_2 - x_1, t_2 - t_1)|^\alpha, 
           \end{align}
           for some constant  $H_e>0$  which is independent of $(x_i, t_i)$. Hence $e \in C^{0, \alpha}(\bar Q)$.  The proof of the theorem is complete.     
\end{proof}

\begin{remark} {\rm Theorem \ref{Theorem-Regularity} is similar to Theorem 6.1 in \cite{Rosch2} but their conclusions and assumptions are somewhat different. Theorem \ref{Theorem-Regularity} concludes that  $e$ is H\"{older} continuous while \cite[Theorem 6.1]{Rosch2} concludes that $g_u[\cdot, \cdot] e$ is H\"{older} continuous. Assumption $(H2)$ requires that $f_y[x,t]\geq C_f$ while $(A5)$ in \cite{Rosch2} required that $f_y[x,t]$ is non-negative. In addition, in Theorem \ref{Theorem-Regularity} the H\"{o}lder continuity of multipliers can be guaranteed when $\Omega$ is convex.} 
\end{remark}

\noindent {\bf Acknowledgment} This research was supported by International Centre for Research and Postgraduate Training in Mathematics, Institute of Mathematics, VAST, under grant number ICRTM02-2022.01.


\begin{thebibliography}{99}



\bibitem{Aubin} J.-P. Aubin and H. Frankowska, {\it Set-valued Analysis},  Birkhauser (1990)

\bibitem{Arada-2002-1} N. Arada and J.-P. Raymond, {\it Dirichlet boundary control of semilinear parabolic equations. Part 1: Problems with no state constraints }, Appl. Math. Optim., 45(2002), 125-143.

\bibitem{Arada-2002-2} N. Arada and J.-P. Raymond, {\it Dirichlet boundary control of semilinear parabolic equations. Part 2: Problems with pointwise state constraints}, Appl. Math. Optim., 45(2002), 145-167.
\bibitem{Binh} T.D. Binh, B.T. Kien, X. Qin and C.-F. Wen, {\it Regularity of multipliers in second-order optimality conditions for
semilinear elliptic control problems}, Applicable Analysis, 101 (2022), 5504-5516.


\bibitem{Bonnans-2000} J.F. Bonnans and A. Shapiro, \textit{Perturbation Analysis of Optimization Problems}, Springer, New York, 2000.

\bibitem{Brezis1}  H. Br\'{e}zis, {\it Functional Analysis, Sobolev spaces and Partial Differential Equations,} Springer, 2010.

\bibitem{Casas-2000} E. Casas, J.-P. Raymond and  H. Zidani, {\it Pontryagin's principle for local solutions of control problems with mixed control-state constraints}, SIAM J. Control Optim., 39(2000), 1182-1203. 


\bibitem{Casas-2015} E. Casas and F. Tr\"{o}ltzsch, {\it Second-order optimality conditions for weak and strong local solutions of parabolic optimal control problems}, Vietnam J. Math., DOI 10.1007/s10013-015-0175-6. 


\bibitem{Casas-2022-1} E. Casas and K. Kunisch, {\it  Optimal control of semilinear parabolic equations with non-smooth pointwise-integral control constraints in time-space}, Appl. Math. Optim. 85, 12 (2022). https://doi.org/10.1007/s00245-022-09850-7.

\bibitem{Casas-2023} E. Casas and D. Wachsmuth, {\it A Note on existence of solutions to control problems of semilinear partial differential equations}, SIAM J. Control Optim., 61 (2023), 1095-1112. 




\bibitem{E.Di.Benedetto}  E. Di Benedetto, \textit{On the local behaviour of solutions of degenerate parabolic equations with measurable coefficients}. Ann. Scuola Sup. Pisa, Ser. I 13 (1986) 487-535.




\bibitem{Evan} L.C. Evan, {\it Partial Differential Equations}, AMS, Providence Rhode Island, 2010. 



\bibitem{Ioffe-1979} A. D. Ioffe and V. M. Tihomirov, \textit{Theory of Extremal Problems}, North-Holland, Amsterdam, 1979.



\bibitem{Kien-2017} B.T. Kien, V.H.  Nhu and N.H. Son, {\it Second-order optimality conditions for a semilinear elliptic optimal control problem with mixed pointwise constraints}, Set-Valued Var. Anal., 25 (2017), 177-210.


\bibitem{Kien-2018} B.T. Kien, N.V.  Tuyen and J.C. Yao, {\it Second-order KKT optimality conditions for multiobjective optimal control problems}, SIAM J. Control Optim.,  56(2018), 4069-4097.

\bibitem{Kien-2022} B.T. Kien and T.D. Binh, {\it On the second-order optimality conditions for multi-objective optimal control problems with mixed pointwise constraints}, J. Global Optim.,  85(2023), 155-183.

\bibitem{Meyer} C. Meyer, U. Pr\"{u}fert and  F. Tr\"{o}ltzsch, {\it On two numerical  methods for state-constrained elliptic control problems}, Optimisation Methods and Software, 22 (2007), 871-899.

\bibitem{Neitzel} I. Neitzel, Tr\"{o}ltzsch, {\it On regularization methods for the numerical solution of parabolic control problems with pointwise state constraints},  ESAIM: Control, Optimisation and Calculus of
Variations, 15(2009), 426-453.









\bibitem{Pales} Z. P\'{a}les and V. Zeidan, {\it Characterization of $L^1$-closed decomposable set in $L^\infty$}, J. Math. Anal. Appl., 238(1999), 491-515. 


\bibitem{Raymond} J.-P. Raymond and H. Zidani, {\it Hamiltonian Pontryagin's principles for control problems governed by semilinear parabolic equations}, Appl. Math. Optim., 39(1999), 143-177.




\bibitem{Rosch1} A. R\"{o}sch  and F.Tr\"{o}ltzsch, {\it  Existence of regular Lagrange multiplier for a nonlinear elliptic optimal control problem with pointwise control-state constraints},  SIAM J Control Optim., 45(2006), 548-564.

\bibitem{Rosch2} A. R\"{o}sch  and F.Tr\"{o}ltzsch, {\it On regularity of solutions and Lagrange multipliers of optimal control problems for semilinear elliptic equations  with pointwise control-state constraints},  SIAM J Control Optim. 46(2007), 1098-1115.

\bibitem{Troltzsch} F. Tr\"{o}ltzsch, {\it Optimal Control of Partial Differential Equations, Theory, Method and Applications}, American Mathematical Society,  Providence Rhode Island, (2010). 




\bibitem{H.Tuy} H. Tuy, {\it Convex Analysis and Global Optimization (Second edition)}, Springer, 2017.

\end{thebibliography}
\end{document}